\documentclass[french]{article}
\usepackage[T1]{fontenc}
\usepackage[latin1]{inputenc}
\usepackage{graphicx}
\usepackage{amssymb}
\usepackage{amsmath}
\usepackage{amsthm}
\usepackage{color}
\usepackage[all]{xy}
\usepackage{amsfonts}
\usepackage{mathrsfs}
\usepackage{latexsym}
\usepackage{geometry}
\usepackage{textcomp}
\usepackage[english,frenchb]{babel}

\theoremstyle{plain}
\newtheorem*{bigtheo}{Théorème}
\newtheorem{theo}{Théorème}[section]
\newtheorem{prop}[theo]{Proposition}
\newtheorem{lemm}[theo]{Lemme}
\newtheorem{ex}[theo]{Exemple}

\newtheorem*{hypo}{Hypothèse}
\theoremstyle{definition}
\newtheorem{defi}[theo]{Définition}
\theoremstyle{remark}
\newtheorem{rema}[theo]{Remarque}

\author{Stéphane Bijakowski}
\title{Classicité de formes modulaires surconvergentes}

\begin{document}

\maketitle

\begin{abstract}
Nous prouvons un résultat de classicité pour les formes modulaires surconvergentes sur les variétés de Shimura PEL de type (A) ou (C) associées à un groupe réductif non ramifié sur $\mathbb{Q}_p$. Nous utilisons pour démontrer ce résultat la méthode du prolongement analytique, initialement développée par Buzzard et Kassaei.
\end{abstract}

\selectlanguage{english}
\begin{abstract}
We prove in this paper a classicality result for overconvergent modular forms on PEL Shimura varieties of type (A) or (C) associated to an unramified reductive group on $\mathbb{Q}_p$. To get this result, we use the analytic continuation method, first used by Buzzard and Kassaei.
\end{abstract}

\selectlanguage{french}

\tableofcontents

\section*{Introduction}
\addcontentsline{toc}{section}{Introduction}

Coleman ($\cite{Co}$) a prouvé qu'une forme modulaire surconvergente sur la courbe modulaire, de niveau iwahorique en $p$, de poids $k$ entier, propre pour un certain opérateur de Hecke $U_p$, était classique si la pente, c'est-à-dire la valuation de la valeur propre pour l'opérateur de Hecke (normalisée de telle façon que $v(p)=1$), était inférieure à $k-1$. Ce résultat a été obtenu grâce à une connaissance approfondie de la cohomologie rigide de la courbe modulaire. Des travaux de Buzzard ($\cite{Bu}$) et de Kassaei ($\cite{Ka}$), utilisant des techniques de prolongement analytique, ont donné une nouvelle démonstration de ce théorème. Ainsi, Buzzard a étudié la dynamique de l'opérateur de Hecke, et montré que ses itérés accumulaient le tube supersingulier dans un voisinage strict arbitrairement petit du lieu ordinaire multiplicatif. Une forme modulaire surconvergente étant définie sur un tel voisinage strict, l'équation $f = a_p^{-1} U_p f$ permet de prolonger $f$ au tube supersingulier dès que la pente est non nulle. \\
\indent La théorie du sous-groupe canonique a ensuite permis à Kassaei de décomposer l'opérateur de Hecke sur le lieu ordinaire-étale (et même sur un voisinage strict de celui-ci) en $U_p = U_p^{good} + U_p^{bad}$, où $U_p^{good}$ paramètre les supplémentaires du sous-groupe universel ne rencontrant pas le sous-groupe canonique, et $U_p^{bad}$ l'unique supplémentaire égal au sous-groupe canonique. L'opérateur~$U_p^{good}$ est à valeurs dans un voisinage strict du tube ordinaire-multiplicatif, donc agit sur les formes modulaires surconvergentes, alors que $U_p^{bad}$ stabilise le lieu ordinaire-étale. Kassaei a alors défini une série $f_n$, approximant $a_p^{-n} U_p^n f$. Plus précisément, on a, lorsque cela a un sens,~$f_n = a_p^{-n} U_p^n f - a_p^{-n} (U_p^{bad})^n f$. La condition $v(a_p) < k-1$ implique alors que la norme de l'opérateur $a_p^{-1} U_p^{bad}$ est strictement inférieure à $1$, et donc que la série $f_n$ converge sur le lieu ordinaire-étale. Cela permet donc d'étendre $f$ à toute la courbe modulaire. \\
\indent Depuis ces travaux, de nombreux résultats de classicité ont été obtenus. La méthode originale de Coleman, qui étudie la cohomologie de la courbe modulaire, a récemment été utilisée par Johansson ($\cite{Jo}$) ainsi que par Tian et Xiao ($\cite{T-X}$) pour obtenir un résultat pour les formes modulaires de Hilbert. La méthode du prolongement analytique a également été utilisée dans plusieurs travaux pour obtenir des résultats de classicité. Citons par exemple les résultats de Sasaki ($\cite{Sa}$), Tian ($\cite{Ti}$) et de Pilloni-Stroh ($\cite{P-S 1}$) pour les formes modulaires de Hilbert. Dans \cite{P-S 2}, Pilloni et Stroh ont obtenu un résultat de classicité pour les formes modulaires sur une variété de Shimura PEL associée à un groupe réductif déployé sur $\mathbb{Q}_p$. Nous démontrons ici un théorème dans le cas où l'algèbre simple de la variété de Shimura sur $\mathbb{Q}_p$ est un produit d'algèbres de matrices à coefficients dans une extension non ramifiée de $\mathbb{Q}_p$. 

\begin{bigtheo}
Soit $p$ un nombre premier, et $X$ une variété de Shimura de type $(A)$ ou $(C)$. On suppose que $p$ est non ramifié dans le corps complexe associé à cette variété, et que sur $\mathbb{Q}_p$, la variété de Shimura est associée à un produit d'algèbres de matrices. Soit $f$ une forme modulaire surconvergente sur $X$ de poids $\kappa$. On suppose que $f$ est propre pour une famille d'opérateurs de Hecke $(U_i)$, de valeurs propres $(a_i)$. Si le poids $\kappa$ est suffisamment grand devant la famille des $(v(a_i))$, alors $f$ est classique.
\end{bigtheo}

Dans le cas des formes modulaires associées à un groupe sur $\mathbb{Q}$ dont la restriction à~$\mathbb{Q}_p$ est  $Res_{\mathbb{Q}_{p^d}/\mathbb{Q}_p} GSp_{2g} $ ($\mathbb{Q}_{p^d}$ est l'unique extension non ramifiée de degré $d$ de $\mathbb{Q}_p$), nous obtenons le théorème suivant.

\begin{bigtheo}
Soient $F$ un corps totalement réel de degré $d$ dans lequel $p$ est inerte et \\
$\Sigma_p=Hom(F,\mathbb{C}_p)$ (où $\mathbb{C}_p$ est la complétion d'une clôture algébrique de $\mathbb{Q}_p$). Soit $f$ une forme modulaire surconvergente de poids $\kappa = (k_{g,i} \leq \dots \leq~k_{1,i})_{i \in \Sigma_p}$ (voir $\ref{formcla}$ pour la définition précise), propre pour $U_p$ (défini dans $\ref{defHecke}$) avec la valeur propre $a_p$. Supposons que 
$$v(a_p) + \frac{dg(g+1)}{2} < \inf_i k_{g,i}$$
Alors $f$ est classique.
\end{bigtheo}

Ce théorème reste vrai si on suppose simplement $p$ non ramifié dans $F$ : il y a alors autant d'opérateurs de Hecke que de places au-dessus de $p$, et donc de conditions sur les pentes pour que la forme modulaire soit classique. \\
Décrivons maintenant la méthode de notre démonstration. La forme modulaire $f$ étant propre pour l'opérateur de Hecke, elle vérifie l'équation fonctionnelle $f= a_p^{-1} U_p f$. Comme dans le cas de la courbe modulaire, nous cherchons à étendre le domaine de définition de $f$ à l'aide de cette équation. Dès que nous rencontrons une zone où le prolongement n'est pas automatique, nous décomposons l'opérateur de Hecke $U_p$ en $U_p^{good} + U_p^{bad}$, où $U_p^{good}$ a son image incluse dans l'espace de définition de $f$, et où $U_p^{bad}$ stabilise la zone en question. Plus précisément, si $x=(A,H)$ est un point de l'espace rigide $X_{rig}$, nous allons distinguer les supplémentaires de $H$, selon que le quotient par ce supplémentaire donne un point où $f$ est déjà définie, ou un point dans la zone qui pose problème. Pour effectuer cette décomposition de l'opérateur de Hecke, nous devons donc stratifier notre zone suivant le nombre de \og mauvais \fg $ $ supplémentaires. Une fois cette décomposition effectuée, nous pouvons alors définir l'analogue des séries utilisées par Kassaei pour étendre le domaine de définition de la forme modulaire. Le fait que ces séries convergent provient de la condition reliant le poids et la pente, qui assure que les normes des opérateurs $a_p^{-1} U_p^{bad}$ soient strictement inférieures à $1$. \\
La méthode utilisée ici est très générale, et permet d'obtenir un résultat similaire pour les variétés de Shimura PEL de type (A) associées à des groupes unitaires. Plus précisément, nous obtenons le théorème suivant.

\begin{bigtheo}
Soit $F_0$ un corps totalement réel de degré $d$ dans lequel $p$ est inerte, $F$ une extension CM de $F_0$ dans lequel $p$ est non ramifié. Nous considérons la variété de Shimura PEL associée à un groupe unitaire relatif à $F/F_0$ de signature $((a,b), \dots, (a,b))$ en $p$ si $p$ est décomposé dans $F$, et de signature $((a,a), \dots, (a,a))$ si $p$ est inerte dans $F$ (dans ce cas on note $b=a$). Soit $\Sigma_p=Hom(F_0,\overline{\mathbb{Q}_p})$  et $f$ une forme modulaire surconvergente de poids $\kappa=~((k_{1,i} \geq \dots \geq k_{a,i}),(l_{1,i} \geq~\dots \geq~l_{b,i}))_{i \in \Sigma_p}$ (voir $\ref{defuni}$ pour la définition précise), propre pour $U_p$, de valeur propre $a_p$. Supposons que 
$$ v(a_p) + dab < \inf_i (k_{a,i} + l_{b,i}) $$
Alors $f$ est classique.
\end{bigtheo}

Là encore, le résultat se généralise au cas où le nombre premier $p$ est non ramifié dans $F$. Il y a alors autant d'opérateurs de Hecke et de conditions à vérifier que de places au-dessus de $p$ dans  $F_0$. Il est également vrai pour les variétés de Shimura PEL de type (A) associées à $F/F_0$ si $p$ est non ramifié dans $F$ (voir le théorème $\ref{theogenuni}$). Dans ces cas, des signatures plus générales sont autorisées, mais pour que le problème ait un sens, nous devons supposer que le lieu ordinaire est non vide. D'après Wedhorn (voir $\cite{We}$), cela est équivalent au fait que $p$ soit totalement décomposé dans le corps réflexe $E$ associé à la variété de Shimura. Cette condition impose des relations sur la signature du groupe unitaire dans le cas (A), ce qui justifie les hypothèses du théorème précédent. 

\begin{rema}
Il ne semble pas y avoir d'obstacle majeur à étendre ces résultats au cas $p$ ramifié. La principale difficulté à contourner serait l'absence du principe de Koecher rigide pour la variété de Shimura. Cela fera l'objet d'un travail ultérieur. 
\end{rema}

\begin{rema}
La méthode développée devrait permettre d'étendre les résultats obtenus au cas d'un niveau arbitraire en $p$. Là encore, la principale difficulté provient de l'absence de modèle entier des variétés considérées, et donc du principe de Koecher rigide. L'auteur espère également résoudre ce problème dans un travail à venir. \\
Remarquons que l'extension du résultat de classicité pour les variétés de niveau arbitraire en $p$ a été obtenu dans le cas Hilbert (\cite{P-S 1}), mais pas dans le cas déployé, car la méthode utilisée dans \cite{P-S 2} ne se généralise pas dans ce cas.
\end{rema}

Parlons à présent de l'organisation du texte. La partie $\ref{prelim}$ énonce des résultats de géométrie rigide, et introduit les objets avec lesquels nous allons travailler. La partie $\ref{premHecke}$ est consacrée à l'étude de la correspondance de Hecke, et à la décomposition de cette correspondance sur certaines zones. La partie $\ref{prolonge}$ démontre le théorème de prolongement analytique pour les variétés de Siegel, après avoir défini les formes modulaires classiques et surconvergentes. Les parties $\ref{symplectique}$ et $\ref{unitaire}$ montrent comment notre démonstration peut s'adapter pour prouver un résultat de classicité pour les variétés de Shimura PEL de type respectivement (C) et (A). \\

L'auteur remercie chaleureusement ses directeurs de thèse B. Stroh et P. Boyer, qui l'ont encadré durant l'élaboration de ce travail. Il souhaite également remercier J. Tilouine et V. Pilloni pour l'organisation de divers séminaires. Enfin, il remercie K. Buzzard et ses étudiants pour leur accueil et leur hospitalité.

\section{Préliminaires} \label{prelim}

\subsection{Géométrie rigide} \label{normdef}

Commençons par énoncer plusieurs résultats de géométrie rigide. Nous suivons pour cela \cite{Ka}. Soit $K$ une ectension finie de $\mathbb{Q}_p$. \\

	Soit $\mathfrak{Z}$ un schéma formel admissible (c'est-à-dire plat et topologiquement de type fini) $p$-adique réduit sur Spf $O_K$, de fibre générique $Z_{rig}$. Soit $\mathcal{F}$ un faisceau localement libre de rang fini sur $\mathfrak{Z}$. Soit $\mathcal{F}_{rig}$ le faisceau induit sur $Z_{rig}$. Soit $L$ une extension finie de $K$  et $x : $ Spm $L \to Z_{rig} $ un $L$-point. Il provient d'un unique $O_L$-point, $\tilde{x} : $ Spf $O_L \to \mathfrak{Z}$. On dispose d'une identification naturelle : $ \tilde{x}^* \mathcal{F} \otimes_{O_L} K \simeq x^* \mathcal{F}_{rig}$. Si $ f \in x^* \mathcal{F}_{rig}$, on pose  $|f| = \inf \{ |\lambda |, \lambda \in K^*, \frac{1}{\lambda} f \in \tilde{x}^* \mathcal{F} \}$.
	
\begin{defi} 
Si $\mathcal{U}$ est un ouvert de $Z_{rig}$, et $f \in H^0(\mathcal{U},\mathcal{F}_{rig})$, on pose 
$$|f(x)| = |x^* f | $$
et $|f|_\mathcal{U} = \sup_{x \in \mathcal{U}} |f(x)|$. Cet élément peut éventuellement être infini, mais $|f|_\mathcal{U} < \infty$ si $\mathcal{U}$ est quasi-compact.
\end{defi}

On notera également $\tilde{\mathcal{F}}_{rig}$ le sous faisceau de $\mathcal{F}_{rig}$ des éléments de norme inférieure ou égale à $1$. Nous énonçons ici un lemme (\og gluing lemma \fg), dû à Kassaei, qui montre qu'une section de $\tilde{\mathcal{F}}_{rig}$ est déterminée par ses réductions modulo $p^n$ pour tout $n$.

\begin{lemm} \label{glue}
Supposons $Z_{rig}$ lisse. Soit $\mathcal{U}$ un ouvert quasi-compact de $Z_{rig}$. On a :
$$H^0(\mathcal{U},\mathcal{F}_{rig}) \simeq H^0(\mathcal{U},\tilde{\mathcal{F}}_{rig}) \otimes_{O_K} K \simeq \left( \underset{\leftarrow}{\lim} \text{ } H^0(\mathcal{U},\tilde{\mathcal{F}}_{rig} / p^n) \right) \otimes_{O_K} K$$
\end{lemm}

\begin{proof}
Voir \cite{Pi}, Corollaire $5.1$.
\end{proof}

Nous rappelons maintenant quelques propri\'et\'es des morphismes finis \'etales d'espaces analytiques rigides. Pour les d\'efinitions, on pourra se reporter par exemple \`a \cite{Be}. Tous les espaces rigides consid\'er\'es ici seront suppos\'es quasi-s\'epar\'es.

\begin{prop} \label{qqcpt}
Soit $f : X \to Y$ un morphisme fini \'etale d'espaces analytiques rigides quasi-compacts. Alors l'image d'un ouvert par $f$ est un ouvert, et l'image d'un ouvert quasi-compact est un ouvert quasi-compact.
\end{prop}

\begin{proof}
D'apr\`es \cite{Bo}, un morphisme plat d'espace analytiques rigides quasi-compacts est ouvert. Le morphisme $f$ \'etant \'etale, donc plat, il envoie un ouvert sur un ouvert. \\
Supposons que $\mathcal{U}$ est un ouvert quasi-compact de $X$, et soit $(\mathcal{V}_i)_{i \in I}$ un recouvrement admissible par des affino\"\i des de $f(\mathcal{U})$. Alors $(f^{-1} (\mathcal{V}_i) \cap \mathcal{U})_{i \in I}$ est un recouvrement admissible de $\mathcal{U}$ ; on peut donc en extraire un sous-recouvrement admissible fini. Il existe donc un ensemble fini $I_0$ tel que $(\mathcal{V}_i)_{i \in I_0}$ recouvre $f(\mathcal{U})$. Comme cet espace est quasi-s\'epar\'e, ce recouvrement est admissible, et cela prouve que l'image de $\mathcal{U}$ par $f$ est quasi-compacte.  
\end{proof}

D\'efinissons maintenant la notion de voisinage strict.

\begin{defi}
Soit $X$ un espace rigide quasi-compact sur $K$, et $\mathcal{U}$ un ouvert quasi-compact. Un voisinage strict de $\mathcal{U}$ est un ouvert quasi-compact $\mathcal{V}$ contenant $\mathcal{U}$ tel que le recouvrement $(\mathcal{V},X \backslash \mathcal{U})$ de $X$ soit admissible.
\end{defi}

\begin{prop}
Soit $X$ un espace rigide sur $K$, $\mathcal{U}_1, \mathcal{U}_2$ des ouverts quasi-compacts de $X$ et soit $\mathcal{V}_i$ un voisinage strict de $\mathcal{U}_i$ pour $i=1,2$. Alors $\mathcal{V}_1 \cap \mathcal{V}_2$ est un voisinage strict de $\mathcal{U}_1 \cap \mathcal{U}_2$, et $\mathcal{V}_1 \cup \mathcal{V}_2$ un voisinage strict de $\mathcal{U}_1 \cup \mathcal{U}_2$.
\end{prop}

\begin{proof}
Traitons par exemple le cas de l'union. Nous devons montrer que le recouvrement $B=(\mathcal{V}_1 \cup \mathcal{V}_2 , X \backslash (\mathcal{U}_1 \cup \mathcal{U}_2))$ de $X$ est admissible. Puisque le recouvrement $(\mathcal{V}_1,X \backslash \mathcal{U}_1)$ de $X$ est admissible, il suffit de v\'erifier que l'intersection du recouvrement $B$ avec $\mathcal{V}_1$ et $X \backslash \mathcal{U}_1$ est admissible. L'intersection de $B$ avec $\mathcal{V}_1$ donne le recouvrement $(\mathcal{V}_1 , \mathcal{V}_1 \backslash (\mathcal{V}_1 \cap ( \mathcal{U}_1 \cup \mathcal{U}_2  )))$ de $\mathcal{V}_1$, qui est admissible. L'intersection avec $X \backslash \mathcal{U}_1$ donne le recouvrement $B' =~((\mathcal{V}_1 \cup~\mathcal{V}_2) \backslash \mathcal{U}_1 , X \backslash (\mathcal{U}_1 \cup~\mathcal{U}_2 )) $ de $X \backslash \mathcal{U}_1$. Or le recouvrement \\
$B'' =~(\mathcal{V}_2 \backslash (\mathcal{V}_2 \cap \mathcal{U}_1) , (X \backslash \mathcal{U}_1) \cap (X \backslash \mathcal{U}_2))$ est un recouvrement admissible de $X \backslash \mathcal{U}_1$ car $\mathcal{V}_2$ est un voisinage strict de $\mathcal{U}_2$. Comme $B''$ est un recouvrement plus fin que $B'$, on en d\'eduit que $B'$ est un recouvrement admissible de $X \backslash \mathcal{U}_1$.
\end{proof}

Les voisinages stricts sont \'egalement pr\'eserv\'es par les morphismes finis \'etales.

\begin{prop} \label{vois}
Soit $f : X \to Y$ un morphisme fini \'etale d'espaces analytiques rigides quasi-compacts sur $K$. Soient $\mathcal{U}$ un ouvert quasi-compact de $X$ et $\mathcal{V}$ un voisinage strict de $\mathcal{U}$ dans $X$. Alors $f(\mathcal{V})$ est un voisinage strict de $f(\mathcal{U})$ dans $Y$.
\end{prop}

\begin{proof}
La propri\'et\'e \`a d\'emontrer \'etant locale, on peut supposer $Y$ affino\"ide. Le fait que le recouvrement $(\mathcal{V}, X \backslash \mathcal{U})$ soit admissible implique qu'il existe un ouvert quasi-compact $\mathcal{U}'$ inclus dans le compl\'ementaire de $\mathcal{U}$ tel que $(\mathcal{V}, \mathcal{U}')$ est un recouvrement admissible de $X$. Choisissons des mod\`eles formels $\mathfrak{X}$, $\mathfrak{Y}$, et $\mathfrak{f} : \mathfrak{X} \to \mathfrak{Y}$, pour $X$, $Y$ et $f$ respectivement. D'apr\`es \cite{BL2} lemme $5.7$, quitte \`a effectuer un \'eclatement admissible de $\mathfrak{X}$, il existe des sous-sch\'emas ouverts $\mathfrak{U}$, $\mathfrak{V}$ et $\mathfrak{U}'$ de $\mathfrak{X}$, dont la fibre g\'en\'erique vaut respectivement $\mathcal{U}$, $\mathcal{V}$ et $\mathcal{U}'$. D'apr\`es \cite{BL2} th\'eor\`eme $5.2$, quitte \`a effectuer un \'eclatement admissible de $\mathfrak{Y}$, et prendre pour $\mathfrak{X}$ le transform\'e strict, on peut supposer que le morphisme $\mathfrak{f}$ est plat. D'apr\`es \cite{BL2} corollaire $5.3$, quitte \`a effectuer un autre \'eclatement admissible de $\mathfrak{Y}$ et prendre pour $\mathfrak{X}$ le transform\'e strict, on peut supposer que $\mathfrak{f}$ est quasi-fini. De plus, $\mathfrak{f}$ est s\'epar\'e d'apr\`es \cite{BL1} proposition $4.7$. Le fait d'\^etre plat \'etant stable par changement de base, on s'est donc ramen\'e au cas o\`u le morphisme $\mathfrak{f}$ est quasi-fini, plat et s\'epar\'e. En raisonnant comme dans \cite{AM} lemme $A.1.1.$, on en d\'eduit que $\mathfrak{f}$ est fini et plat. Soit $X_0$ la fibre sp\'eciale de $\mathfrak{X}$, et de m\^eme pour $Y_0$, $\mathcal{U}_0$, $\mathcal{V}_0$ et $\mathcal{U}'_0$. Alors $\mathcal{U}$, $\mathcal{V}$ et $\mathcal{U}'$ sont \'egaux \`a l'image inverse de leur fibre sp\'eciale par le morphisme de sp\'ecialisation. De plus, $\mathcal{V}_0$ et $\mathcal{U}'_0$ sont deux ouverts recouvrant $X_0$, et on a 
$$\mathcal{U}_0 \subset X_0 \backslash \mathcal{U}'_0 \subset \mathcal{V}_0  $$
Le morphisme $f$ induit sur les fibres sp\'eciales \'etant fini et plat, les images de $\mathcal{U}_0$ et $\mathcal{V}_0$ par $f$ sont des ouverts de $Y_0$, et l'image de $X_0 \backslash \mathcal{U}'_0$ est un ferm\'e $\mathcal{W}_0$ de $Y_0$ (le morphisme $f$ est propre, donc ferm\'e). On a alors
$$f(\mathcal{U}_0) \subset \mathcal{W}_0 \subset f(\mathcal{V}_0)    $$
Les ouverts $f(\mathcal{U})$ et $f(\mathcal{V})$ sont \'egaux aux images inverses de $f(\mathcal{U}_0)$ et $f(\mathcal{V}_0)$ par le morphisme de sp\'ecialisation. Or si $\mathcal{W}'_0$ d\'esigne le compl\'ementaire de $\mathcal{W}_0$, et $\mathcal{W}'$ l'image inverse de $\mathcal{W}'_0$ par le morphisme de sp\'ecialisation, alors $\mathcal{W}'$ est un ouvert quasi-compact contenant le compl\'ementaire de $f(\mathcal{V})$ et contenu dans le compl\'ementaire de $f(\mathcal{U})$. Cela implique que les recouvrements $(f(\mathcal{V}),\mathcal{W}')$ et $(f(\mathcal{V}),Y \backslash f(\mathcal{U}))$ sont admissibles, donc que $f(\mathcal{V})$ est un voisinage strict de $f(\mathcal{U})$.
\end{proof}

\'Enon\c cons maintenant un crit\`ere de finitude, qui sera tr\`es important pour d\'ecomposer les op\'erateurs de Hecke.

\begin{prop} \label{finitefiber}
Soient $X,Y$ deux espaces rigides s\'epar\'es, et $f : X \to Y$ un morphisme fini \'etale. Soit \'egalement $\mathcal{U}$ un ouvert de $X$ telle que l'immersion ouverte $\mathcal{U} \to X$ soit quasi-compacte. Supposons que le cardinal des fibres g\'eom\'etriques de $f_{|\mathcal{U}} : \mathcal{U} \to Y$ soit constant. Alors les morphismes $f_{|\mathcal{U}} : \mathcal{U} \to Y$ et $f_{|X\backslash \mathcal{U}} : X \backslash \mathcal{U} \to Y$ sont finis et \'etales.
\end{prop}

\begin{proof}
La propri\'et\'e \'etant locale sur $Y$, on peut supposer $Y$ affino\"ide. Alors \\
$X=~f^{-1} (Y)$ est \'egalement affino\"ide, $\mathcal{U}$ est quasi-compact. Soit $\mathfrak{X}$, $\mathfrak{Y}$ et $\mathfrak{f} : \mathfrak{X} \to \mathfrak{Y}$ des mod\`eles formels pour $X$, $Y$ et $f$ respectivement. En raisonnant comme dans la proposition pr\'ec\'edente, quitte \`a effectuer un \'eclatement admissible de $\mathfrak{X}$, il existe un sous-sch\'ema ouvert $\mathfrak{U}$ de $\mathfrak{X}$ \'egal \`a $\mathcal{U}$ en fibre g\'en\'erique. De m\^eme, quitte \`a effectuer un \'eclatement admissible de $\mathfrak{Y}$, on peut supposer que le morphisme $\mathfrak{f} : \mathfrak{X} \to \mathfrak{Y}$ est  s\'epar\'e, quasi-fini et plat. Le morphisme $\mathfrak{U} \to \mathfrak{Y}$ est \'egalement s\'epar\'e, quasi-fini et plat. Le lemme A.1.1. de \cite{AM} montre alors que les morphismes $\mathfrak{X} \to \mathfrak{Y}$ et $\mathfrak{U} \to \mathfrak{Y}$ sont finis. \\
Nous allons maintenant montrer que le morphisme $\mathfrak{X} \backslash \mathfrak{U} \to \mathfrak{Y}$ est propre. V\'erifions le crit\`ere valuatif de propret\'e. Soit $R$ un anneau de valuation de corps des fractions $L$. Soit $y_0 :$~Spec~$R \to~\mathfrak{Y}$ un morphisme fix\'e. Notons $y : $ Spec $L \to \mathfrak{Y}$ le morphisme induit par $y_0$. Comme le morphisme $\mathfrak{X} \to \mathfrak{Y}$ est fini, il existe un nombre fini de points g\'eom\'etriques au-dessus de $y$. Soit $L'$ une extension finie telle que les points g\'eom\'etriques au-dessus de $y$ soient des morphismes $x_1, \dots, x_d : $ Spec $L' \to \mathfrak{X}$ faisant commuter le diagramme

\begin{displaymath}
\xymatrix{\text{Spec } L' \ar[r]^-{x_i} \ar[d] &  \mathfrak{X} \ar[d] \\
\text{Spec } R' \ar[r]^-{y_0} & \mathfrak{Y}
}
\end{displaymath}

\noindent o\`u $R'$ est la cl\^oture int\'egrale de $R$ dans $L'$. Chacun de ces morphismes se rel\`eve de mani\`ere unique en $\widetilde{x}_i : $ Spec $R' \to \mathfrak{X}$. Supposons que les morphismes $x_1, \dots, x_r$ soient \`a valeurs dans $\mathfrak{U}$ et $x_{r+1}, \dots, x_{d}$ \`a valeurs dans $\mathfrak{X} \backslash \mathfrak{U}$. Alors par propret\'e du morphisme $\mathfrak{U} \to \mathfrak{Y}$, les morphismes $\widetilde{x_1}, \dots, \widetilde{x_r}$ sont \`a valeurs dans $\mathfrak{U}$. Comme le cardinal des fibres de $\mathfrak{U} \to \mathfrak{Y}$ est constant, les morphismes $\widetilde{x_{r+1}}, \dots, \widetilde{x_d}$ sont \`a valeurs dans $\mathfrak{X} \backslash \mathfrak{U}$. Nous avons donc prouv\'e que tout morphisme $x_0 :$ Spec $L \to \mathfrak{X} \backslash \mathfrak{U}$ faisant commuter le diagramme

\begin{displaymath}
\xymatrix{\text{Spec } L \ar[r]^-{x_0} \ar[d] &  \mathfrak{X} \backslash \mathfrak{U} \ar[d] \\
\text{Spec } R \ar[r]^-{y_0} & \mathfrak{Y}
}
\end{displaymath}

\noindent s'\'etend en un unique morphisme $\widetilde{x_0} : $ Spec $ R \to \mathfrak{X} \backslash \mathfrak{U}$. Cela prouve que le morphisme $\mathfrak{X} \backslash \mathfrak{U} \to \mathfrak{Y}$ induit par $\mathfrak{f}$ est propre. On en d\'eduit qu'il est fini d'apr\`es \cite{EGA_3} Proposition 4.4.2. \\
Cela implique donc que les morphismes $\mathcal{U} \to Y$ et $X \backslash \mathcal{U} \to Y$ au niveau des espaces rigides sont finis. Il sont \'egalement \'etales, car les morphismes d'inclusion $\mathcal{U} \to X$ et $X \backslash \mathcal{U} \to X$ le sont.
\end{proof}

\begin{rema} 
Cette proposition est l'analogue d'un r\'esultat de g\'eom\'etrie alg\'ebrique (voir \cite{EGA_IV_4} cor. $18.2.9$). Le lemme $A.1.1$ de \cite{AM} permet de montrer directement que la restriction de $f$ \`a $\mathcal{U}$ est finie (mais pas la restriction de $f$ \`a $X \backslash \mathcal{U}$, car il faut d'abord prouver que l'immersion $X \backslash \mathcal{U} \to X$ est quasi-compacte).
\end{rema}

\subsection{Correspondances cohomologiques}

Dans cette partie, nous allons introduire la définition de correspondance cohomologique sur un espace rigide. Cela permet de définir à la fois une correspondance géométrique, et un opérateur agissant sur les sections de certains faisceaux. Nous utiliserons cette définition pour définir des opérateurs de Hecke.

\begin{defi}
Soit $X,Y,Z$ des espaces rigides. Une correspondance géométrique entre $X$ et $Y$ en dessous de $Z$ est la donnée de deux morphismes étales $p : Z \to X$ et $q : Z \to Y$. A toute partie $S$ de $X$ on associe la partie $C(S):=q(p^{-1}(S))$ de $Y$.
\end{defi}

Cette correspondance envoie les ouverts sur les ouverts. Si $q$ est de plus finie, elle envoie les ouverts quasi-compacts sur les ouverts quasi-compacts.

\begin{defi}
Soient deux faisceaux cohérents $\mathcal{F}$ et $\mathcal{G}$ définis respectivement sur $X$ et $Y$ et $c : q^* \mathcal{G} \to p^* \mathcal{F}$ un morphisme de faisceaux cohérents. Si $p$ est de plus fini, le morphisme $c$ permet de définir un morphisme $H^0(C(\mathcal{U}),\mathcal{G}) \to H^0(\mathcal{U},\mathcal{F})$, pour tout ouvert $\mathcal{U}$ de $X$. Il est défini par
\begin{displaymath}
H^0(C(\mathcal{U}),\mathcal{G}) \to H^0 ( p^{-1} (\mathcal{U}), q^* \mathcal{G}) \overset{c}{\to} H^0(p^{-1}(\mathcal{U}) , p^* \mathcal{F}) \overset{Tr_{p}}{\to}  H^0(\mathcal{U},\mathcal{F})
\end{displaymath}
Ce morphisme est la correspondance cohomologique.
\end{defi}

\subsection{Degré et degrés partiels} \label{partial}

Dans $\cite{Fa}$, Fargues définit une fonction degré pour les schémas en groupes finis et plats, qui nous sera très utile. Nous en rappelons ici la définition et les principales propriétés. \\
Soit $K$ un extension finie de $\mathbb{Q}_p$, $O_K$ son anneau des entiers et $\pi$ une uniformisante. Soit également $\overline{K}$ une clôture algébrique de $K$. Soit $v$ la valuation sur $\overline{K}$ normalisée par $v(p)=1$.

\begin{defi}
Soit $G$ un schéma en groupes fini et plat sur $O_K$. Nous définissons le degré de $G$ par
$$\text{deg }(G) = v(\text{Fitt}_0 \text{ } \omega_G)$$
où
\begin{itemize}
\item $\omega_G$ est le faisceau conormal de $G$.
\item Fitt$_0$ désigne l'idéal de Fitting.
\item La valuation d'un idéal principal $I=(a)$ est définie par $v(I):=v(a)$.
\end{itemize}
\end{defi}

\noindent Ainsi, si $\omega_G = \oplus_{i=1}^d O_K / \pi^{a_i} O_K$, alors deg$(G) = \sum_{i=1}^d a_i$.

\begin{ex}
Si $G$ est un schéma en groupes de Oort-Tate sur $O_K$ ($\cite{T-O}$) de paramètres $(a,b)$, de telle sorte que $G = $ Spec $O_K[X]/(X^p - aX)$, alors deg $G = v(a)$.
\end{ex}

De plus, cette fonction a les propriétés suivantes (\cite{Fa}) :

\begin{prop} \label{propdeg}
La fonction deg vérifie les propriétés suivantes :
\begin{itemize}
\item Si $0 \to G_1 \to G_2 \to G_3 \to 0$ est une suite exacte, alors deg $G_2 = $ deg $G_1 + $ deg $G_3$.
\item deg $G$ + deg $G^D = $ht $G$, où $G^D$ est le dual de Cartier de $G$ et ht $G$ la hauteur de $G$.
\item Si $f : G \to G'$ est un morphisme de schémas en groupes finis et plats sur $O_K$, qui induit un isomorphisme en fibre générique, alors deg $G \leq $ deg $G'$. De plus, il y a égalité si et seulement $f$ est un isomorphisme.
\end{itemize}
\end{prop}

\begin{rema} 
$ $
\begin{itemize}
\item La deuxième propriété implique en particulier que deg $G \leq $ ht $G$. 
\item Si $G$ est de hauteur $h$, alors deg $G = 0$ si et seulement si $G$ est étale, c'est-à-dire isomorphe sur $O_{\overline{K}}$ à $(\mathbb{Z}/p\mathbb{Z})^h$. De même, deg $G=h$ si et seulement si $G$ est multiplicatif, c'est-à-dire isomorphe sur $O_{\overline{K}}$ à $\mu_p^h$.
\end{itemize}
\end{rema}

Définissons maintenant les degrés partiels pour les schémas en groupes finis et plats munis d'une action de l'anneau des entiers d'une extension finie non ramifiée de $\mathbb{Q}_p$. Nous appliquerons en particulier ces résultats pour les sous-groupes finis d'un groupe $p$-divisible. \\
Soit $F$ une extension finie non ramifiée de $\mathbb{Q}_p$ de degré $f$, et $O_F$ son anneau des entiers. On a donc $O_F = W(\mathbb{F}_{p^f})$ et $F=O_F[1/p]$. Soit $S$ l'ensemble des plongements de $F$ dans $\overline{\mathbb{Q}_p}$ ; on sait que $S$ est un groupe cyclique d'ordre $f$ engendré par le Frobenius. \\
Soit $K$ une extension finie de $\mathbb{Q}_p$ contenant $F$, et soit $H$ un schéma en groupes fini et plat d'ordre une puissance de $p$ sur $O_K$ muni d'une action de $O_F$ de hauteur $fh$. Soit $\omega_H$ le module des différentielles ; c'est un $O_K$-module de type fini muni d'une action de $O_F$. Alors, on a 
$$\omega_H = \bigoplus_{s \in S} \omega_{H,s}$$
où $\omega_{H,s}$ est le sous-module de $\omega_H$ où $O_{F}$ agit par $s$. 

\begin{defi}
Le degré partiel de $H$ relatif au plongement $s$ de $F$ est défini par 
$$\text{deg}_s H := v( \text{Fitt}_0 \text{ } \omega_{H,s})$$
\end{defi}

\begin{ex}
Si $G$ est un schéma en groupes de Raynaud sur $O_K$ ($\cite{Ray}$) de paramètres $(a_i,b_i)$, de telle sorte que $G = $ Spec $O_K[X_1, \dots, X_f]/(X_i^p - a_iX_{i+1})$, alors $\deg_i G =~v(a_{i-1})$.
\end{ex}

On voit immédiatement que le degré de $H$ est égal à la somme des deg$_s H$ pour $s \in S$. Nous allons maintenant démontrer des propriétés analogues à la fonction degré pour les degrés partiels.

\begin{prop} 
Les fonctions deg$_s$ sont additives. Plus précisément, soient $H_1$, $H_2$ et $H_3$ trois groupes finis et plats d'ordre une puissance de $p$ munis d'une action de $O_F$ avec une suite exacte
$$ 0 \to H_1 \to H_2 \to H_3 \to 0$$
Alors pour tout $s \in S$
$$\text{deg}_s H_2 = \text{deg}_s H_1 + \text{deg}_s H_3 $$
\end{prop}

\begin{proof}
On a une suite exacte de $O_K \otimes_{\mathbb{Z}_p} O_F$-modules
$$0 \to \omega_{H_3} \to \omega_{H_2} \to \omega_{H_1} \to 0$$
En décomposant cette suite exacte suivant les éléments de $S$, on en déduit des suites exactes
$$0 \to \omega_{H_3,s} \to \omega_{H_2,s} \to \omega_{H_1,s} \to 0$$
pour tout $s \in S$. Le résultat en découle.
\end{proof}

\begin{prop} \label{dual}
Soit $H$ un schéma en groupes fini et plat de d'ordre une puissance de $p$ sur $O_K$ muni d'une action de $O_F$ de hauteur $fh$. Soit $H^D$ le dual de Cartier de $H$ ; c'est encore un schéma en groupes fini et plat sur $O_K$ muni d'une action de $O_F$. Alors pour tout $s \in S$, 
$$\text{deg}_s H^D = h - \text{deg}_s H$$
En particulier, on voit que deg$_s H \in [0 , h]$. 
\end{prop}

\begin{proof}
On se ramène au cas où $H$ est de $p$-torsion. Soit $(\mathfrak{M},\phi)$ le module de Breuil-Kisin de $H$ (voir $\cite{Ki}$) ; $\mathfrak{M}$ est un $k[[u]]$-module libre de rang $fh$ et $\phi$ est un endomorphisme semi-linéaire tel que $u^e \mathfrak{M}$ soit inclus dans le module engendré par l'image de $\phi$, où $k$ est le corps résiduel de $O_K$ et $e$ son indice de ramification. Le module $\mathfrak{M}$ est muni d'une action de $O_F$, donc se décompose suivant les éléments de $S$ : $\mathfrak{M} = \oplus_{s \in S} \mathfrak{M}_s$. On choisit une bijection entre $S$ et $\mathbb{Z} / f \mathbb{Z}$ de telle sorte que $\phi$ envoie $\mathfrak{M}_i$ dans $\mathfrak{M}_{i+1}$. Les $\mathfrak{M}_i$ sont donc des $k[[u]]$-modules libres de rang $h$. On note $\phi_i : \mathfrak{M}_{i-1} \to \mathfrak{M}_i$. Fixons une base pour les modules $(\mathfrak{M}_i)$, et soit $A_i$ la matrice de $\phi_i$ dans cette base. On a alors
$$ \text{deg}_i H = \frac{1}{e} v_u(\det A_i)$$
où $v_u$ dénote la valuation $u$-adique. De plus, le module de Breuil-Kisin de $H^D$ est $(\mathfrak{M}^*,\phi^*)$, où $\mathfrak{M}^*$ est le dual de $M$, et où $\phi^*$ peut être décrit comme suit. Le module $\mathfrak{M}^*$ se décompose en $\mathfrak{M}^* = \oplus_{i=0}^{f-1} \mathfrak{M}_i^*$. On muni chaque module $\mathfrak{M}_i^*$ de la base duale de celle des $\mathfrak{M}_i$. Alors la matrice de $\phi_i^* : \mathfrak{M}_{i-1}^* \to \mathfrak{M}_i^*$ dans cette base est $B_i = {u^e} ({}^t\! A_i) ^{-1}$. D'où
$$ \text{deg}_i H^D = \frac{1}{e} v_u(\det B_i) = \frac{1}{e} (eh - v_u(\det A_i)) = h - \text{deg}_i H $$
Une autre démonstration possible aurait été de filtrer le groupe $H$ par des groupes de Raynaud (voir $\cite{Ray}$), et d'utiliser l'additivité des fonctions degrés (la propriété est évidente pour les groupes de Raynaud car on a une description explicite de ces groupes et de leurs duaux).
\end{proof}

Nous avons également une propriété d'augmentation des degrés partiels par déformation. On fixe une identification entre $S$ et le groupe $\mathbb{Z}/f\mathbb{Z}$, le Frobenius arithmétique étant identifié avec $1$. 

\begin{prop}
Soient $G$ et $G'$ deux schémas en groupes finis et plats sur $O_K$ d'ordre une puissance de $p$ munis d'une action de $O_F$. On suppose qu'il existe un morphisme $O_F$-linéaire $f : G \to G'$, qui est un isomorphisme en fibre générique. Alors pour tout $j \in \mathbb{Z}/f\mathbb{Z}$, on a 
$$\sum_{i=0}^{f-1} p^i \deg_{j-i} G \leq \sum_{i=0}^{f-1} p^i \deg_{j-i} G'$$
De plus, $f$ est un isomorphisme si et seulement si $\deg_j G = \deg_j G'$ pour tout $j \in \mathbb{Z}/f\mathbb{Z}$.
\end{prop}

\begin{proof}
On se ramène au cas où $G$ et $G'$ sont de $p$-torsion. Soient $(\mathfrak{M},\phi)$ et $(\mathfrak{M}',\phi')$ les modules de Breuil-Kisin associés respectivement à $G$ et $G'$. On a donc un morphisme $f : \mathfrak{M}' \to \mathfrak{M}$. Chacun de ses modules se décompose sous l'action de $O_F$ : on a donc $\mathfrak{M} = \oplus_{i \in \mathbb{Z}/f\mathbb{Z}} \mathfrak{M}_i$, tels que le Frobenius $\phi$ envoie $\mathfrak{M}_i$ dans $\mathfrak{M}_{i+1}$, et de même pour $\mathfrak{M}'$. Fixons des bases pour les $k[[u]]$-modules libres $\mathfrak{M}_i$, et soit $A_i$ la matrice de $\phi_i : \mathfrak{M}_{i-1} \to \mathfrak{M}_i$. On a donc comme précédemment 
$$ \text{deg}_i G = \frac{1}{e} v_u(\det A_i)$$
On fait de même pour les modules $\mathfrak{M}_i'$. Le morphisme $f$ induit des morphismes $f_i :~\mathfrak{M}_i' \to~\mathfrak{M}_i$, tels que $f_i \circ \phi_i' = \phi_i \circ f_{i-1}$. Si on note $F_i$ la matrice de $f_i$ dans les bases introduites, on a donc
$$F_i A_i' = A_i \sigma(F_{i-1})$$
En prenant le déterminant, on obtient
$$ \det F_i \det A_i' = (\det F_{i-1})^p \det A_i$$
On en déduit que pour tout $j \in \mathbb{Z}/f\mathbb{Z}$, on a
$$\det F_j \prod_{i=0}^{f-1} (\det A_{j-i}')^{p^{i}} = (\det F_j)^{p^f} \prod_{i=0}^{f-1} (\det A_{j-i})^{p^{i}} $$ 
En passant à la valuation $u$-adique, et en utilisant le fait que la valuation $u$-adique de $\det F_j$ est positive, on obtient les relations
$$\sum_{i=0}^{f-1} p^i \deg_{j-i} G \leq \sum_{i=0}^{f-1} p^i \deg_{j-i} G'$$
Supposons maintenant que $\deg_j G = \deg_j G'$ pour tout $j \in \mathbb{Z}/f\mathbb{Z}$. Alors les inégalités précédentes sont des égalités, ce qui prouve que tous les $f_i$ sont des isomorphismes, soit que $G$ et $G'$ sont isomorphes. \\
Ici encore, une autre démonstration possible aurait été de filtrer les groupes $G$ et $G'$ par des groupes de Raynaud. En effet, il est possible de filtrer le groupe $G$ par des sous-groupes $(H_i)$ tels que $H_{i+1} / H_i$ soit un sous-groupe de Raynaud. Il existe également une filtration $(H_i')$ de $G'$ vérifiant la même propriété telle que le morphisme $f$ induise des morphismes $f_i : H_i / H_{i-1} \to H_i' / H_{i-1}'$ vérifiant les hypothèses de la proposition. En utilisant l'additivité des fonctions degrés, il suffit de prouver la proposition pour les groupes de Raynaud, ce qui est fait dans $\cite{P-S 1}$.
\end{proof}

\section{Variétés de Hilbert-Siegel} \label{premHecke}

Dans les deux parties suivantes, nous allons étudier le cas des variétés de Hilbert-Siegel. Commençons par quelques définitions.

\subsection{L'espace de modules}

Soit $F$ une extension totalement réelle de $\mathbb{Q}$ de degré $d$, $O_F$ son anneau des entiers et $p$ un nombre premier inerte dans $F$. On note $F_p$ la complétion $p$-adique de $F$, et $k \simeq \mathbb{F}_q$ le corps résiduel, avec $q=p^d$. On a donc $F_p = $Frac $W(k)$, où $W$ signifie les vecteurs de Witt. \\
Soit $g$ un entier strictement positif, et posons $K=F_p$ pour simplifier les notations. Soit $N \geq 3$ un entier premier à $p$. Notons également $\delta$ la différente de $F$. \\

\begin{defi}
Soit $X$ l'espace de modules sur $O_K$ dont les $S$-points sont les classes d'isomorphismes des quintuplets $(A,\lambda,\iota,\eta, H_i)$ avec
\begin{itemize}
\item $A \to S$ est un schéma abélien de dimension $dg$.
\item $\lambda : A \to A^t$ est une polarisation principale $O_F$-linéaire.
\item $\iota : O_F \hookrightarrow $ End$(A)$ est un morphisme compatible avec l'involution de Rosati.
\item $\eta : \delta \otimes_{\mathbb{Z}} \mu_N \hookrightarrow A[N]$ est une structure principale de niveau $N$ compatible à $O_F$.
\item $H_1 \subset \dots \subset H_g \subset A[p]$ est un drapeau complet de la $p$-torsion, où $H_i$ est un sous-groupe fini et plat totalement isotrope stable par $O_F$ de rang $p^{di}$.
\end{itemize}
\end{defi}

L'espace de modules $X$ est représentable par un schéma quasi-projectif sur $O_K$, que l'on notera toujours $X$. On notera également $X_K = X \times K$, $\mathfrak{X}$ la complétion formelle de $X$ le long de sa fibre spéciale, et $X_{rig}$ la fibre générique rigide de $\mathfrak{X}$. 

\begin{rema}
Le schéma $X$ est défini sur $\mathbb{Z}_p$, mais ce ne sera pas le cas des faisceaux que nous définirons dans les sections ultérieures.
\end{rema} 

\begin{defi}
On définit l'application deg : $X_{rig} \to [0,dg]$ par deg $(A,\lambda,\iota,\eta, H_i):=~\deg H_g$.
\end{defi}

Si $I$ est un intervalle inclus dans $[0,dg]$, on note $X_I = $ deg$^{-1} (I)$. On note également $X_{\geq u} = $ deg$^{-1} ([u, +\infty[)$ et $X_{> u}~=~$~deg$^{-1} (]u, +\infty[)$ . 

\begin{prop} \label{degree}
Si $I$ est un intervalle inclus dans $[0,dg]$, alors $X_I$ est un ouvert de $X_{rig}$. Si $I$ est compact à bornes rationnelles, alors $X_I$ est quasi-compact. 
\end{prop}

\begin{proof}
Sur $X$ l'isogénie universelle $A \to A/H_g$ induit un morphisme $\omega_{A/H_g} \to~\omega_A$, où $\omega_A$ et $\omega_{A/H_g}$ sont les faisceaux conornaux associés respectivement à $A$ et $A/H_g$. Ce sont des faisceaux localement libres de rang $dg$ sur $X$. Soient $\omega_{A/H_g}'$ et $\omega_A'$ les déterminants de ces faisceaux ; ce sont des faisceaux inversibles sur $X$ et on a donc un morphisme $\omega_{A/H_g}' \to \omega_A'$. Ce morphisme induit donc une section $\delta_{H_g} \in H^0(X, \mathcal{L}_{H_g})$, où $\mathcal{L}_{H_g}$ est le faisceau inversible égal à $\omega_A' \otimes \omega_{A/H_g}'^{-1}$. \\
On notera encore $\mathcal{L}_{H_g}$ le faisceau inversible sur $X_{rig}$ ; on dispose encore d'une section $\delta_{H_g} \in~H^0(X_{rig}, \mathcal{L}_{H_g})$. D'après le paragraphe $\ref{normdef}$, pour tout point $x$ de $X_{rig}$ on peut définir la norme de $\delta_{H_g} (x)$. De plus, en utilisant la définition de la fonction degré, on voit que
$$| \delta_{H_g} (x) | = p^{- \deg H_g(x)}$$
La fonction deg s'exprime donc comme valuation d'une fonction analytique sur $X_{rig}$. La proposition en découle, puisque si $I=[a,b]$ par exemple, alors
$$X_I = \{x \in X_{rig} , p^{-b} \leq |\delta_{H_g} (x) | \leq p^{-a} \} $$
et on voit donc que $X_I$ est un ouvert quasi-compact. Si $I$ n'est plus un intervalle compact, il faut remplacer certaines inégalités larges par des inégalités strictes, ce qui montre que $X_I$ est un ouvert de $X_{rig}$. 
\end{proof}

\subsection{Correspondance de Hecke}

Soit $C_K$ l'espace de modules sur $K$, dont les $S$-points sont les couples $(A,\lambda,\iota,\eta, H_i,L)$ où $(A,\lambda,\iota,\eta, H_i) \in X_K(S)$, et $L$ est un supplémentaire de $H_g$ dans $A[p]$, totalement isotrope et stable par l'action de $O_F$. \\
On dispose de deux morphismes de $C_K$ vers $X_K$ : $p_1 : (A,\lambda,\iota,\eta, H_i,L) \to (A,\lambda,\iota,\eta, H_i)$ et $p_2 :~(A,\lambda,\iota,\eta, H_i,L) \to (A/L , \iota',\lambda',\eta', $Im $(H_i \to (A/L)[p]))$, où $\iota' : O_F \to$ End$(A/L)$ est induit par $\iota$, $\lambda': A/L \to (A/L)^t$ est induit par $p \cdot\lambda$, et $\eta'$ est induit par $\eta$. Soit $C^{an}$ l'analytifié de $C_K$, et $X^{an}$ l'analytifié de $X_K$. On note encore $p_1, p_2$ les morphismes $C^{an} \to X^{an}$. On note $C_{rig} = p_1^{-1} (X_{rig})$ ; c'est le lieu de bonne réduction, et on  a également $C_{rig} = p_2^{-1} (X_{rig})$. 

\begin{defi} L'opérateur de Hecke ensembliste sur $X_{rig}$ est défini sur chaque partie $S$ de $X_{rig}$ comme $U_p(S) = p_2 p_1^{-1} (S)$. 
\end{defi}

\begin{rema} 
$ $
\begin{itemize}
\item Celui-ci envoie les parties finies dans les parties finies, les ouverts Zariski dans les ouverts Zariski, et les ouverts admissibles quasi-compacts dans les ouverts admissibles quasi-compacts. 
\item Si $x=(A,\lambda,\iota,\eta, H_i) \in X_{rig}$, l'ensemble des points de $U_p(x)$ est donc en bijection avec les sous-groupes $L$ de $A[p]$, totalement isotropes et stables par l'action de $O_F$, tels que $L$ et $H_g$ soient des supplémentaires dans $A[p]$ sur $K$. 
\end{itemize}
\end{rema}

Dans la suite, nous appellerons simplement \og supplémentaire générique \fg \text{ }de $H_g$ un tel $L$. L'opérateur de Hecke augmente le degré. Plus précisément, on a :

\begin{prop}
Soient $K_1$ une extension finie de $K$, $x=(A,\lambda,\iota,\eta, H_i) \in X_{rig}(K_1)$ et $y \in~U_p(x)$. Alors deg $(x) \leq $ deg $(y)$ ; de plus si deg $(x) = $ deg $(y)$, alors
\begin{itemize}
\item $H_g$ est un Barsotti-Tate tronqué d'échelon $1$.
\item il existe une extension finie $K_2$ de $K_1$, et un sous-groupe $H' \subset A_{O_{K_2}} [p]$, de hauteur $dg$ et stable par $O_F$, tels que $H_g$ et $H'$ soient en somme directe dans $A_{O_{K_2}} [p]$.
\item deg $(H') = dg - $ deg $(H_g)$.
\item deg $(H')$ et deg $(H)$ sont des entiers.
\end{itemize}
\end{prop}

\begin{proof}
Soit $H'$ le supplémentaire générique de $H_g$ associé à $y$. Il est défini sur une extension finie $K_2$ de $K_1$. L'application $H_g \times H' \to A[p]$ est un isomorphisme en fibre générique ; d'après la propriété $\ref{propdeg}$, on a donc 
deg $(H_g) + $ deg $(H') \leq dg$. Or $\deg(x) = \deg (H_g)$, et deg $(y) = $ deg $(A[p] / H')~=~dg - $deg~$(H')$. D'où deg $(x) \leq $ deg $(y)$. \\
Si l'inégalité précédente est une égalité, toujours d'après la propriété $\ref{propdeg}$, l'application $H_g \times~H' \to~A[p]$ est isomorphisme, ce qui donne les propriétés énoncées.
\end{proof}

\begin{rema}
On a en particulier $U_p (X_{\geq u}) \subset X_{\geq u}$ pour tout $u \in [0,dg]$.
\end{rema}

Comprendre la dynamique de l'opérateur de Hecke $U_p$ revient à distinguer deux cas, suivant que le degré est entier ou non. Le cas où le degré n'est pas un entier est bien compris.

\begin{prop} \label{dyna}
Soit $r$ un entier compris entre $0$ et $dg -1$. Soit $0 < \lambda < \mu < 1$ des réels. Alors, il existe un entier $N$ tel que
$$U_p^N (X_{\geq r + \lambda}) \subset X_{\geq r + \mu}$$
\end{prop}

\begin{proof}
En effet, supposons par l'absurde qu'il existe $x_n \in X_{\geq r + \lambda}$ et $y_n \in U_p^n(x_n)$ avec deg $(y_n) < r + \mu$. D'après la proposition précédente, cela entraîne $x_n \in X_{[r+\lambda,r+\mu]}$. Or cet espace est quasi-compact, et l'opérateur $U_p$ augmente strictement la fonction degré sur cet espace. Nous allons montrer que l'augmentation du degré peut être minorée par un élément strictement positif. Définissons 
$$C_{[r+\lambda,r+\mu]} = p_1^{-1} (X_{[r+\lambda,r+\mu]})$$
De plus, comme dans le preuve de la proposition $\ref{degree}$, il existe un faisceau inversible $\mathcal{L}_{H_g}$ sur $X_{rig}$, et une section $\delta_{H_g}$ de ce faisceau, tel que $|\delta_{H_g} (x)| = p^{-\deg(x)}$ pour tout point $x$ de $X_{rig}$. On note encore $\mathcal{L}_{H_g}$ le faisceau inversible sur $C_{rig}$ égal à $p_1^* \mathcal{L}_{H_g}$, et $\delta_{H_g}$ la section de ce faisceau égal à $p_1^* \delta_{H_g}$. La norme pour le faisceau $\mathcal{L}_{H_g}$ sur $X_{rig}$ donne par transport une norme pour le faisceau $p_1^* \mathcal{L}_{H_g}$ sur $C_{rig}$. De même, à l'aide de la projection $p_2$, on définit le faisceau $\mathcal{L}_{H_g'} = p_2^* \mathcal{L}_{H_g}$, la section $\delta_{H_g'} = p_2^* \delta_{H_g}$, ainsi que le sous-faisceau des fonctions de norme inférieure ou égale à $1$. Ainsi, si $K_1$ est une extension finie de $K$, et $x$ un $K_1$-point de $C_{rig}$, correspondant à un couple $(A\lambda,\iota,\eta, H_i)$ défini sur $O_{K_1}$, ainsi qu'un sous-groupe $L$ de $A[p]$, alors
\begin{displaymath}
\begin{array}{ccc}
| \delta_{H_g} (x) | = p^{- \deg H_g(x)} &
\text{ et } &
| \delta_{H_g'} (x) | = p^{- \deg H_g'(x)}
\end{array}
\end{displaymath}
où $H_g'$ est l'image de $H_g$ dans $A/L$. Nous avons vu que $|\delta_{H_g'} |< |\delta_{H_g}|$ sur $C_{[r+\lambda,r+\mu]}$ ; comme ce dernier espace est quasi-compact, on a $|\delta_{H_g'} | \leq \alpha |\delta_{H_g}|$ avec $\alpha <1$. Il existe donc $\varepsilon > 0$, avec
$$ \text{deg}(y) \geq \text{ deg}(x) + \varepsilon$$
pour tout $x \in X_{[r+\lambda,r+\mu]}$ et $y \in U_p(x)$. \\
Cela implique donc que deg$(y_n) \geq n\varepsilon + $ deg$(x_n) \geq n \varepsilon +r + \lambda$ pour tout $n$, ce qui est impossible.
\end{proof}

\subsection{Décomposition de $U_p$} \label{zoneint}

Sur $X_{]dg-1,dg]}$, l'opérateur de Hecke contracte donc les points vers le lieu ordinaire-multiplicatif. En dehors de cette zone, ce n'est pas le cas : il existe des points $x$ de degré inférieur à $dg-1$ tels que $U_p^n (x)$ ne soit pas inclus dans un voisinage strict du lieu ordinaire multiplicatif fixé pour tout entier $n$. Nous allons séparer les points de $U_p^n(x)$ suivant qu'ils sont ou non proches du lieu ordinaire multiplicatif. \\
Pour tout rationnel $\alpha < 1$, nous allons effectuer un découpage de l'ouvert $\mathcal{U}=\mathcal{U}_\alpha =~X_{[0, dg - 1 + \alpha]}$ qui permettra de décomposer l'opérateur de Hecke. 

\begin{defi} \label{bonnesuite}
Soit $\mathcal{V}$ un ouvert quasi-compact de $X_{rig}$. Une bonne suite d'ouverts de $\mathcal{V}$ est une suite d\'ecroissante $(\mathcal{V}_k)_{k\geq 0}$ d'ouverts quasi-compacts, avec $\mathcal{V}_0=\mathcal{V}$ et $\mathcal{V}_k = \emptyset$ si $k$ est plus grand qu'un certain entier $R$. La longueur de la suite est le plus petit entier $R$ v\'erifiant cette propri\'et\'e.
\end{defi}

	Nous allons donc découper l'ouvert $\mathcal{U}$ suivant le nombre de \og mauvais \fg $ $ supplémentaires génériques, c'est-à-dire ceux de degré plus grand que $1$. En effet, si on quotiente par un tel supplémentaire, le degré sera toujours inférieur à $dg - 1$, et on ne quitte donc pas l'ouvert $\mathcal{U}$. 
	
\begin{defi}
Soit $\beta$ un rationnel avec $0 < \beta < 1$. Pour tout $x=(A,\lambda,\iota,\eta, H_i) \in \mathcal{U}$, soit $N(x,\beta)$ le nombre de supplémentaires génériques $L$ de $H_g$, avec deg$(L) \geq 1 - \beta$. 
\end{defi}

Bien sûr la fonction $N$ ne peut prendre qu'un nombre fini de valeurs. Soit $N_{max}$ le nombre de supplémentaires génériques de $H_g$ dans $A[p]$ ; alors $N(x,\beta) \leq N_{max}$. 

\begin{rema} 
Si $x=(A,\lambda,\iota,\eta, H_i) \in \mathcal{U}$ et $L$ est un supplémentaire générique de $H_g$ avec deg~$L < 1-\beta$, alors
$$\text{deg }(A[p]/L) = dg - \text{deg } L > dg-1 + \beta$$
Si $y=(A/L,\iota',\lambda',\eta', H_i')$, alors deg $(y) > dg - 1+\beta$, soit $y \in X_{> dg - 1+\beta}$. Les fonctions $N(\bullet,\beta)$ vont nous servir à compter le nombre de mauvais supplémentaires génériques de $H_g$. 
\end{rema}

\begin{defi}
Soit $\mathcal{U}_i$ l'espace défini par
$$\mathcal{U}_i := \{ x \in \mathcal{U} , N(x,\beta) \geq i \}$$
\end{defi}

\begin{lemm} \label{bonnesuiteprop}
Les $(\mathcal{U}_i)_{i\geq 0}$ forment une bonne suite d'ouverts de $\mathcal{U}$.
\end{lemm}

\begin{proof}
Il est clair que $\mathcal{U}_0 = \mathcal{U}$, et que les $(\mathcal{U}_i)$ forment une suite décroissante, avec $\mathcal{U}_i = \emptyset$ si $i>N_{max}$. \\
Notons pour tout $i \geq 1$ l'espace de modules $C_i$ sur $K$, dont les $S$-points sont les couples $(A,\lambda,\iota,\eta, H_k, L_1, \dots, L_i)$, avec $(A,\lambda,\iota,\eta, H_k) \in X_K(S)$, et où les $L_k$ sont des supplémentaires génériques distincts de $H_g$. Soit $p_i : C_i \to X_K$ la projection d'oubli des $L_j$ ; on note $C_i^{an}$ l'analytifié de $C_i$ et $C_{i,rig} = p_i^{-1} (X_{rig})$. \\
Soit $C_{i,\beta}$ le sous-espace de $C_{i,rig}$ formé des $(A,\lambda,\iota,\eta,H_k,L_j)$ avec $(A,\lambda,\iota,\eta,H_k) \in \mathcal{U}$ et deg~$(L_j) \geq~1 -~\beta$ pour tout $1 \leq j \leq i$. C'est un ouvert quasi-compact de $C_{i,rig}$, et $\mathcal{U}_i$ est l'image par le morphisme fini et étale $p_i$ de $C_{i,\beta}$, donc $\mathcal{U}_i$ est un ouvert quasi-compact de $\mathcal{U}$, d'après la proposition $\ref{qqcpt}$.
\end{proof}

\begin{prop}
Soit $0 \leq i \leq N_{max}$. Sur $\mathcal{U}_i \backslash \mathcal{U}_{i+1}$, on a $N(x,\beta) = i$, et on peut décomposer la correspondance géométrique de Hecke sur cet ensemble en $U_{p} = U_{p,i}^{good} \coprod U_{p,i}^{bad}$, où $U_{p,i}^{bad}$ correspond aux $i$ supplémentaires $L$ de degré supérieur ou égal à $1 - \beta$.
\end{prop}

\begin{proof}
Justifions que cette correspondance est bien finie et étale. Rappelons que la correspondance de Hecke utilise l'espace rigide $C_{rig}$ paramétrant les couples  $(A,\lambda,\iota,\eta,H_k,L)$ avec $(A,\lambda,\iota,\eta,H_k)$ un point de $X_{rig}$ et $L$ supplémentaire générique de $H_g$. On dispose de deux morphismes finis et étales $p_1, p_2 : C_{rig} \to X_{rig}$, où $p_1$ est défini par l'oubli de $L$ et $p_2$ par le quotient par $L$. Notons $X_i = \mathcal{U}_i \backslash \mathcal{U}_{i+1}$, $C_i = p_1^{-1} (X_i)$ et $D_{i,\beta} \subset C_i$ le lieu (ouvert) défini par deg $(L) \geq 1 - \beta$. Alors l'immersion ouverte $D_{i,\beta} \to C_i$ est quasi-compacte, et les fibres géométriques de $p_1 : D_{i,\beta} \to X_i$ sont de cardinal constant (égal à $i$). D'après la proposition $\ref{finitefiber}$, les restriction de $p_1$ à $D_{i,\beta}$ et $C_i \backslash D_{i,\beta}$ sont finies et étales. \\
De plus, les morphismes $p_2 : D_{i,\beta} \to X_{rig}$ et $p_2 : C_i \backslash D_{i,\beta} \to X_{rig}$ sont étales car les immersions ouvertes le sont. On peut donc décomposer sur $X_i$ l'opérateur $U_p$ en $U_{p} =~U_{p,i}^{good} \coprod U_{p,i}^{bad}$, où $U_{p,i}^{bad}$ est obtenu à partir de la restriction de $p_1$ et $p_2$ à $D_{i,\beta}$, et où $U_{p,i}^{good}$ est obtenu en restreignant $p_1$ et $p_2$ à $C_i \backslash D_{i,\beta}$.
\end{proof}

\begin{rema}
$ $
\begin{itemize}
\item Pour $i=0$, on a sur $\mathcal{U} \backslash \mathcal{U}_1$, $U_{p,0}^{bad} = \emptyset$, et $U_p = U_{p,0}^{good}$. 
\item Pour $i=N_{max}$, on a sur $\mathcal{U}_{N_{max}}$, $U_{p,N_{max}}^{good} = \emptyset$ et $U_p = U_{p,N_{max}}^{bad}$, c'est-à-dire que tous les supplémentaires sont mauvais.
\item L'image de $U_{p,i}^{good}$ est incluse dans $X_{>dg-1+\beta}$, et l'image de $U_{p,i}^{bad}$ est incluse dans $X_{\leq dg - 1 + \beta}$ pour tout $i$.
\end{itemize}
\end{rema}

Nous aurons également besoin de faire surconverger les ouverts $\mathcal{U}_i$. Soit $\beta'$ un rationnel avec $\beta < \beta ' < 1$. 

\begin{defi}
On définit l'ouvert $\mathcal{U}_i'$ par
$$\mathcal{U}_i' := \{ x \in \mathcal{U} , N(x,\beta') \geq i \}$$
\end{defi}

\noindent La suite d'ouverts $(\mathcal{U}_i')$ vérifie alors les mêmes propriétés que celles de $(\mathcal{U}_i)$ ; en particulier on peut également décomposer l'opérateur de Hecke sur $\mathcal{U}_{i}' \backslash \mathcal{U}_{i+1}'$, qui coïncide avec la décomposition de $U_p$ sur $(\mathcal{U}_{i} \backslash \mathcal{U}_{i+1}) \cap (\mathcal{U}_{i}' \backslash \mathcal{U}_{i+1}') = \mathcal{U}_{i} \backslash (\mathcal{U}_i \cap \mathcal{U}_{i+1}')$.

\begin{prop} \label{bonnesuitevois}
Pour tout $k \geq 1$, $\mathcal{U}_k'$ est un voisinage strict de $\mathcal{U}_k$.
\end{prop}

\begin{proof}
Nous reprenons les notations des deux démonstrations précédentes. Nous avons un morphisme fini étale $\textrm{p}_k : C_{k,rig} \to X_{rig}$. De plus, $\mathcal{U}_k = \textrm{p}_k (C_{k,\beta})$ et $\mathcal{U}_k'=\textrm{p}_k(C_{k,\beta'})$. Montrons que $C_{k,\beta'}$ est un voisinage strict de $C_{k,\beta}$ dans $\textrm{p}_k^{-1} (\mathcal{U})$. \\
Rappelons que $C_k$ paramètre les $(A,\lambda,\iota,\eta, H_i, L_1, \dots, L_k)$, où les $L_j$ sont des supplémentaires génériques de $H_g$ deux à deux disjoints. L'ouvert quasi-compact $C_{k,\beta}$ est défini par les conditions $(A,\lambda,\iota,\eta, H_k) \in \mathcal{U}$ et deg $L_j \geq 1 - \beta$ pour tout $1 \leq j \leq k$. Ces dernières conditions sont équivalentes à $|\delta_{L_j} (x) | \leq p^{-1+\beta}$, où $\delta_{L_j}$ est la section du fibré inversible associé à $L_j$. \\
On en déduit donc que le recouvrement $(C_{k,\beta'},\textrm{p}_k^{-1} (\mathcal{U}) \backslash C_{k,\beta})$ est admissible puisque $\beta' > \beta$, soit que $C_{k,\beta'}$ est un voisinage strict de $C_{k,\beta}$. \\
Le morphisme $\textrm{p}_k$ étant fini et étale, $\mathcal{U}_k'$ est un voisinage strict de $\mathcal{U}_k$ d'après la proposition $\ref{vois}$.
\end{proof}

\subsection{Décomposition de $U_p^N$} \label{Hecke}

Dans la section précédente, nous avons découpé l'ouvert $\mathcal{U}$ suivant le nombre de mauvais supplémentaires, et avons ainsi décomposé la correspondance de Hecke en chaque point de $\mathcal{U}$ en $U_p = U_p^{good} + U_p^{bad}$, avec $U_p^{good}$ d'image incluse dans $X_{> dg - 1 + \beta}$, et $U_p^{bad}$ d'image incluse dans $X_{\leq dg - 1 + \beta}$. Il est possible d'itérer cette correspondance, et de décomposer la correspondance de Hecke d'ordre $p^N$, $U_p^N$. On rappelle que l'on s'est donné un rationnel $\alpha <1$ et que $\mathcal{U}=X_{[0, dg-1+ \alpha]}$.

\begin{theo} \label{bigtheo}
Soit $N \geq 1$ et $\beta$ un rationnel avec $0<\beta<1$. Il existe un ensemble fini totalement ordonné $S_N$ (qui sera défini par récurrence), indépendant de $\alpha$ et $\beta$, et une bonne suite d'ouverts $(\mathcal{U}_i (N))_{i \in S_N}$ de $\mathcal{U}=X_{[0, dg-1+ \alpha]}$ de longueur $L=L(N)$ indépendante de $\alpha$ et $\beta$, tels que pour tout $i\geq 0$, on peut décomposer la correspondance $U_p^N$ sur $\mathcal{U}_{i}(N) \backslash \mathcal{U}_{i+1}(N)$ en 
$$ U_{p}^N = \left ( \coprod_{k=0}^{N-1} U_p^{N-1-k} \circ T_k  \right ) \coprod T_N$$
avec $T_0 = U_{p,i,N}^{good}$, pour $0 < k < N$
$$T_k = \coprod_{i_1 \in S_{N-1}, \dots, i_k \in S_{N-k}}  U_{p,i_k,N}^{good} U_{p,i_{k-1},i_k,N}^{bad} \dots U_{p,i,i_1,N}^{bad}$$
et
$$T_N = \coprod_{i_1 \in S_{N-1}, \dots, i_{N-1} \in S_1} U_{p,i_{N-1},N}^{bad} U_{p,i_{N-2},i_{N-1},N}^{bad} \dots U_{p,i,i_1,N}^{bad}$$  
avec
\begin{itemize}
\item les images des opérateurs $U_{p,j,N}^{good}$ ($j \in S_k$) sont incluses dans $X_{> dg - 1 + \beta}$
\item les opérateurs $U_{p,i,j,N}^{bad}$ ($i \in S_k$, $j \in S_{k-1})$ et $U_{p,j,N}^{bad}$ ($j \in S_1$) sont obtenus en quotientant par un sous-groupe $L$ de degré supérieur ou égal à $1 - \beta$, et ont donc leurs images incluses dans $X_{\leq dg-1 + \beta}$.
\end{itemize}
Enfin, si $\beta'$ est un autre rationnel avec $\beta < \beta' <1$, et si $(\mathcal{U}_i'(N))$ est la bonne suite d'ouverts obtenue pour $\beta'$, alors $\mathcal{U}_i'(N)$ est un voisinage strict de $\mathcal{U}_i(N)$ pour tout $i$.
\end{theo}

\begin{proof}
Nous allons démontrer ce résultat par récurrence, ce qui permettra de construire explicitement les ensembles $S_k$ ainsi que les différents opérateurs intervenant dans la preuve. \\
Pour $N=1$, le résultat est démontré dans le paragraphe précédent. Remarquons que l'on a $S_1 = \{ 0, 1, \dots, N_{max} \}$. Supposons le résultat vrai pour $N \geq 1$ et démontrons le pour $N+1$.\\ 
Soient donc $\alpha < 1$, $\mathcal{U}=X_{[0,dg-1+\alpha]}$. Dans le paragraphe précédent, nous avons décomposé sur chaque cran de la bonne suite d'ouverts de $\mathcal{U}$ l'opérateur $U_p$ en $U_p^{good} \coprod U_p^{bad}$, où l'image de $U_p^{good}$ est incluse dans $X_{>dg - 1 + \beta}$, et l'opérateur $U_p^{bad}$ est obtenu en quotientant par des sous-groupes de degré supérieur ou égal à $1 - \beta$. L'image de $U_p^{bad}$ est donc incluse dans $X_{[0 , dg-1 + \beta]} =: \mathcal{V}$. \\
On peut appliquer le théorème au rang $N$ à $\mathcal{V}$ : il existe une suite décroissante d'ouverts quasi-compacts $(\mathcal{V}_i)_{i \in S_N}$ de $\mathcal{V}$, avec $\mathcal{V}_0=\mathcal{V}$, $\mathcal{V}_{L+1} = \emptyset$ si $L+1$ est le cardinal de $S_N$, et une décomposition de $U_{p,i}^N$ sur $\mathcal{V}_i \backslash \mathcal{V}_{i+1}$ pour $i \in S_N$. \\
Nous devons maintenant découper l'ouvert $\mathcal{U}$, suivant non seulement le nombre de mauvais supplémentaires, mais également suivant l'image dans $\mathcal{V}$ de ces mauvais supplémentaires. Ainsi, nous allons compter le nombre de supplémentaires, non plus dans $\mathcal{V}$ (c'est ce qui a été fait pour la décomposition de $U_p$), mais dans chaque cran $\mathcal{V}_i \backslash \mathcal{V}_{i+1}$. \\
Soit $S_{N+1}$ l'ensemble des suites décroissantes d'entiers positifs ou nuls $N_{max} \geq m_0 \geq~\dots \geq~m_{L}$, où $L+1$ est le cardinal de $S_N$. On posera $m_{-1} = N_{max}$. L'ensemble $S_{N+1}$ s'identifie donc à l'ensemble des fonctions sur $S_N$, à valeurs dans $\mathbb{N}$, décroissantes et majorées par $N_{max}$. \\
C'est bien un ensemble fini ; de plus, nous pouvons ordonner cet ensemble par l'ordre lexicographique : si $\underline{n} = (n_i)$ et $\underline{m}=(m_i)$ sont deux telles suites distinctes, et si $i$ est le premier indice tel que $n_i \neq m_i$, on dit que $\underline{n} < \underline{m}$ si $n_i < m_i$. \\
Si $\underline{m} = (m_i)$, le successeur de $\underline{m}$ est alors défini de la façon suivante : soit $i\geq -1$ le plus grand entier tel que $m_i > m_{i+1}$ (qui existe bien si $\underline{m}$ n'est pas la suite constante égale à $N_{max}$), alors 
$$\underline{m}+1 = (m_0, \dots, m_i, m_{i+1} +1, 0, \dots, 0)$$
Définissons maintenant la bonne suite d'ouverts de $\mathcal{U}$. Pour tout $x=(A,\lambda,\iota,\eta,H_k) \in \mathcal{U}$ et $i \in S_N$, soit $N_i (x)$ le nombre de supplémentaires génériques $L$ de $H_g$ avec \\
$y_L=(A/L,\iota',\lambda',\eta', $Im$ (H_k \to (A/L)[p])) \in \mathcal{V}_i$. Remarquons que le fait que $y_L$ appartienne à $\mathcal{V}_i$ implique que deg $(y_L) \leq dg-1+\beta$, soit deg $(L) \geq 1 - \beta$. En particulier, les fonctions $N_i$ sont bornées par $N_{max}$. De plus, $N_i(x)$ est également le cardinal de $U_p(x) \cap \mathcal{V}_i$. \\
Soit
$$\mathcal{U}_{\underline{m}} := \bigcap_{j\in S_N} \left( \bigcup_{i<j} \{ x \in \mathcal{U}, N_i(x) \geq m_i +1 \} \cup \{ x \in \mathcal{U}, N_j(x) \geq m_j \} \right)$$ 
Si $\underline{m}=(N_{max}, \dots, N_{max})$, définissons $\mathcal{U}_{\underline{m}+1}$ comme étant l'ensemble vide. \\
Avant de poursuivre la démonstration du théorème, démontrons quelques propriétés de ces ensembles.

\begin{lemm}
Pour tout $\underline{m} \in S_{N+1}$, $\mathcal{U}_{\underline{m}}$ est un ouvert quasi-compact.
\end{lemm}

\begin{proof}
Les intersections et unions intervenant dans la définition de $\mathcal{U}_{\underline{m}}$ étant finies, il suffit de montrer que $\mathcal{U}_{k,l} := \{ x \in \mathcal{U}, N_l(x) \geq k \}$ est un ouvert quasi-compact de $\mathcal{U}$.  \\
Soit $C_k$ l'espace des modules des $(A,\lambda,\iota,\eta,H_j,L_1, \dots, L_k)$ avec $(A,\lambda,\iota,\eta,H_j) \in \mathcal{U}$ et $L_1, \dots, L_k$ des supplémentaires génériques distincts de $H_g$. On dispose du morphisme d'oubli des $L_j$, $p :~C_k \to~\mathcal{U}$, ainsi que d'un morphisme $q : C_k \to X^k$, la $j$-ième composante étant donnée par le quotient par $L_j$. Alors
$$\mathcal{U}_{k,l} = p(q^{-1} ( \mathcal{V}_l^k))$$
Comme $\mathcal{V}_l$ est un ouvert quasi-compact, et que $p$ et $q$ sont finis et étales, $\mathcal{U}_{k,l}$ est également un ouvert quasi-compact.
\end{proof}

De plus, il n'est pas difficile de voir que la suite $(\mathcal{U}_{\underline{m}})$ est décroissante, donc définit une bonne suite d'ouverts de $\mathcal{U}$. Décomposons maintenant la correspondance de Hecke sur chacune des crans.

\begin{lemm}
Pour tout $\underline{m} \in S_{N+1}$, et $x \in \mathcal{U}_{\underline{m}} \backslash \mathcal{U}_{\underline{m}+1}$, on a $N_i(x)=m_i$
\end{lemm}

\begin{proof}
Si $\underline{m}$ est égal à la suite constante égale à $N_{max}$, le résultat est clair, car les fonctions $N_i$ sont bornées par $N_{max}$. \\
Soit $k \geq -1$ le plus grand entier tel que $m_k>m_{k+1}$. On a donc $m_j=m_{k+1}$ pour tout $j> k+1$. \\
Définissons pour $j \in S_N$
$$R_j = \bigcup_{i<j} \{ x \in \mathcal{U}, N_i(x) \geq m_i +1 \} \cup \{ x \in \mathcal{U}, N_j(x) \geq m_j \}$$
Soit $x \in \mathcal{U}_{\underline{m}} \backslash \mathcal{U}_{\underline{m}+1}$  : il est donc dans $R_j$ pour tout $j \in S_N$. De plus, $\mathcal{U}_{\underline{m}+1}$ est l'intersection des $R_j$ pour $j \leq k$ et de 
$$\bigcup_{i<k+1} \{ x \in \mathcal{U}, N_i(x) \geq m_i +1 \} \cup \{ x \in \mathcal{U}, N_{k+1}(x) \geq m_{k+1}+1 \}$$
Comme $x \notin \mathcal{U}_{\underline{m}+1}$, il n'est pas dans ce dernier ensemble et on a $N_i(x) \leq m_i$ pour $i \leq k+1$. De plus, pour $i > k+1$, on a $N_i(x) \leq N_{k+1} (x) \leq m_{k+1} = m_i$, donc $N_i(x) \leq m_i$ pour tout $i \in S_N$. \\
Le fait que $x$ soit dans $R_i$ donne alors $N_{i}(x) = m_i$ pour tout $i \in S_N$.
\end{proof}

Finissons maintenant la démonstration du théorème $\ref{bigtheo}$. Soit $\underline{m} \in S_{N+1}$ ; alors pour tout $x \in \mathcal{U}_{\underline{m}} \backslash \mathcal{U}_{\underline{m}+1}$, il y a exactement $m_i$ points dans $U_p(x) \cap \mathcal{V}_i$, pour tout $i \in S_N$. Il y a donc $m_i - m_{i+1}$ points dans $U_p(x) \cap (\mathcal{V}_i \backslash \mathcal{V}_{i+1})$, et $N_{max}-m_0$ points dans $U_p(x) \cap (X_{rig} \backslash \mathcal{V}_0)$. On peut alors décomposer l'opérateur de Hecke sur $\mathcal{U}_{\underline{m}} \backslash \mathcal{U}_{\underline{m}+1}$ en  
$$U_{p} = U_{p,\underline{m}}^{good} \coprod_{i \in S_N} U_{p,\underline{m},i}^{bad}$$
où l'opérateur $U_{p,\underline{m}}^{good}$ correspond aux $N_{max} - m_0$ points dans $X_{rig} \backslash \mathcal{V}_0 = X_{> dg - 1 + \beta}$ et $U_{p,\underline{m},i}^{bad}$ aux $m_i - m_{i+1}$ points dans $\mathcal{V}_i \backslash \mathcal{V}_{i+1}$. En effet, d'après la proposition $\ref{finitefiber}$, sur $\mathcal{U}_{\underline{m}} \backslash \mathcal{U}_{\underline{m}+1}$, on peut décomposer l'opérateur $U_p$ en $U_{p,\underline{m}}^{good} \coprod U_{p,\underline{m}}^{bad}$, où $U_{p,\underline{m}}^{good}$ correspond aux $N_{max} - m_0$ points dans $X_{rig} \backslash \mathcal{V}_0$ et $U_{p,\underline{m}}^{bad}$ aux $m_0$ points dans $\mathcal{V}_0$. En appliquant à nouveau la proposition $\ref{finitefiber}$, on voit qu'on peut décomposer $U_{p,\underline{m}}^{bad}$ en $U_{p,\underline{m},0}^{bad} \coprod \widetilde{U_{p,\underline{m},0}^{bad}}$, où $U_{p,\underline{m},0}^{bad}$ correspond aux $m_0 - m_1$ points dans $\mathcal{V}_0 \backslash \mathcal{V}_{1}$ et $\widetilde{U_{p,\underline{m},0}^{bad}}$ aux $m_1$ points dans $\mathcal{V}_1$. En itérant ce processus, on obtient bien la décomposition de $U_p$ annoncée. \\
Or par récurrence, nous avons une décomposition de $U_p^N$ sur $\mathcal{V}_i \backslash \mathcal{V}_{i+1}$, ce qui donne une décomposition de $U_p^{N+1}$ sur $\mathcal{U}_{\underline{m}} \backslash \mathcal{U}_{\underline{m}+1}$. \\
L'opérateur $U_{p,\underline{m}}^{good}$ est bien d'image incluse dans $X_{> dg - 1 + \beta}$, et les opérateurs $U_{p,\underline{m},i}^{bad}$ correspondent à des supplémentaires de degré supérieur ou égal à $1 - \beta$, donc ont leur image incluse dans $X_{\leq dg - 1 + \beta}$. \\
Enfin, il possible de faire surconverger la décomposition précédente en prenant $ \beta < \beta' < 1$, auquel cas $\mathcal{V}'=X_{[0,dg-1+\beta']}$ est un voisinage strict de $\mathcal{V}$. Par récurrence $\mathcal{V}_i'$ est un voisinage strict de $\mathcal{V}_i$ pour tout $i$, et on obtient donc que $\mathcal{U}_{\underline{m}}'$ est un voisinage strict de $\mathcal{U}_{\underline{m}}$ pour tout $\underline{m}$, où $(\mathcal{U}_{\underline{m}}')$ est la bonne suite d'ouverts obtenue pour $\beta'$.

\end{proof}

Nous avons décomposé l'opérateur $U_p^N$ sur chaque cran de la suite d'ouverts en utilisant par récurrence la décomposition obtenue pour $U_p^{N-1}$.  Il est en fait possible d'obtenir une autre suite d'ouverts en utilisant uniquement l'opérateur $U_{p^N}$, qui est égal à $U_p^N$. Rappelons la définition de cet opérateur. \\
Soit $C_{N,K}$ l'espaces de modules sur $K$ dont les $S$-points sont les couples $(A,\lambda,\iota,\eta, H_i,L)$ où $(A,\lambda,\iota,\eta, H_i) \in X_K(S)$, et $L$ est un sous-groupe de $A[p^N]$ de rang $Ndg$, totalement isotrope, stable par l'action de $O_F$ et disjoint de $H_g$. On dispose de deux morphismes $p_1$ et $p_2$ de $C_{N,K}$ vers $X_K$ : $p_1$ est l'oubli de $L$ et $p_2$ est le quotient par $L$. On note $C_N^{an}$ l'analytifié de $C_{N,K}$, et $C_{N,rig} = p_1^{-1} (X_{rig})$. 

\begin{defi}
L'opérateur de Hecke géométrique $U_{p^N}$ agissant sur les parties de $X_{rig}$ est défini par $S \to p_2(p_1^{-1}(S))$.
\end{defi}

On vérifie alors facilement que $U_{p^N} (S) = U_p^N (S)$ pour toute partie $S$ de $X_{rig}$. On peut alors appliquer les résultats de la section précédente en remplaçant l'opérateur $U_p$ par $U_{p^N}$. On obtient le résultat suivant.

\begin{theo} \label{autre_decompo}
Soit $\beta$ un rationnel compris entre $0$ et $1$. Il existe une bonne suite d'ouverts $(\mathcal{V}_i(N))_i$ de longueur $N_{max}^N$, tel que si $x \in \mathcal{V}_i(N) \backslash \mathcal{V}_{i+1}(N) $, le cardinal de $U_{p^N}(x) \cap X_{\leq dg-1+\beta}$ est égal à $i$. Sur $\mathcal{V}_i(N) \backslash \mathcal{V}_{i+1}(N)$, on peut décomposer l'opérateur $U_{p^N}$ en
$$U_{p^N} = U_{p^N,i}^{good} \coprod U_{p^N,i}^{bad}$$
où $U_{p^N,i}^{good}$ a son image incluse dans $X_{> dg-1+\beta}$, et $U_{p^N,i}^{bad}$ a son image incluse dans $X_{\leq dg-1+\beta}$. \\
Si $\beta' > \beta$ est un autre rationnel, et si $(\mathcal{V}_i(N)')_i$ est la suite d'ouverts obtenue pour $\beta'$, alors $\mathcal{V}_i(N)$ est un voisinage strict de $\mathcal{V}_i(N)$ pour tout $i$.
\end{theo}

\begin{proof}
La construction est entièrement analogue à celle effectuée dans le paragraphe précédent, en remplaçant l'opérateur $U_p$ par $U_{p^N}$.
\end{proof}

Le lien entre les deux suites d'ouverts, et les décompositions de $U_p^N$ est fait dans la proposition suivante.

\begin{prop}
Soit $\beta$ un rationnel avec $0< \beta <1$, et soit $(\mathcal{U}_i(N))_{i \in S_N}$ et $(\mathcal{V}_i(N))_{0 \leq i \leq N_{max}^N}$ les deux bonnes suites d'ouverts obtenues précédemment. Alors sur $\mathcal{U}_i(N) \backslash \mathcal{U}_{i+1}(N)$, le cardinal de $U_{p^N}(x) \cap X_{\leq dg-1+\beta}$ est constant ; on a donc $\mathcal{U}_i(N) \backslash \mathcal{U}_{i+1}(N) \subset \mathcal{V}_j(N) \backslash \mathcal{V}_{j+1}(N)$ pour un certain entier $j$. On a de plus
$$U_{p^N,j}^{good} =  \coprod_{k=0}^{N-1} U_p^{N-1-k} \circ T_k $$
$$U_{p^N,j}^{bad} = T_N$$
avec les notations précédentes pour $T_k$ et $T_N$, cette égalité étant vraie sur $\mathcal{U}_i(N) \backslash \mathcal{U}_{i+1}(N)$.
\end{prop}

\begin{proof}
Le fait que sur un cran $\mathcal{U}_i(N) \backslash \mathcal{U}_{i+1}(N)$ le nombres de mauvais supplémentaires soit constant découle de la construction : ce nombre est égal à 
$$\sum_{i_1 \in S_{N-1}, \dots, i_{N-1} \in S_1} N_{i_{N-1}} N_{i_{N-2},i_{N-1}} \dots N_{i,i_1}$$
où $N_{i_k,i_{k+1}}$ est le nombre de points de $U_{p,i_{k},i_{k+1},N}^{bad}$. \\
L'égalité des décompositions provient du fait que tous les opérateurs \og good \fg $ $ ont leur image incluse dans $X_{\leq dg-1+\beta}$, tous les opérateurs \og bad \fg $ $ ont leur image dans $X_{< dg-1+\beta}$, et que ces espaces sont disjoints.
\end{proof}

La première décomposition de $U_p^N$ est donc plus fine que la seconde. Nous verrons dans la suite que les deux décompositions auront leur utilité.

\section{Prolongement analytique des formes modulaires de Hilbert-Siegel} \label{prolonge}

\subsection{Formes modulaires classiques et surconvergentes}  \label{formcla}

Soit le schéma en groupes $\mathcal{G}=Res_{O_{F_p}/\mathbb{Z}_p} GL_g$, $B$ le Borel supérieur, $U$ son radical unipotent, et $T$ son tore maximal. Soit $X(T)$ le groupe des caractères de $T$ sur $\overline{K}$. 

\begin{rema} 
Les $\overline{K}$-points de $\mathcal{G}$ sont $GL_g (\overline{K} \otimes_{\mathbb{Z}_p} O_{F_p} )$, et $\overline{K} \otimes_{\mathbb{Z}_p} O_{F_p}$ s'identifie à $\overline{K}^d$ par le morphisme $x \otimes f\to~(x \sigma_i(f))$, où les $\sigma_i$ sont les éléments de $Hom_{\mathbb{Q}_p} (F_p,\overline{K})$.
\end{rema}

\noindent \`A toute famille $\kappa= (k_{j,i})_{1 \leq j \leq g, 1 \leq i \leq d}$, on associe le caractère de $T$
$$(x_1 \otimes f_1, \dots, x_g \otimes f_g) \to \prod_{i,j} (x_j \sigma_i (f_j))^{k_{j,i}}$$
L'application précédente donne donc un isomorphisme $(\mathbb{Z}^g)^d \simeq X(T)$. \\

\noindent \textbf{\textit{Notation :}}
Si $\kappa = (k_{j,i})$ correspond à un caractère, on note $\kappa' = (-k_{g+1-j,i})$. \\

\begin{defi}
Soit $R$ un anneau sur $O_{F_p}$, $A \to $ Spec $R$ une variété abélienne de dimension $g$, et $e$ sa section unité. Soit $ \omega : (R \otimes_{\mathbb{Z}_p} O_{F_p} )^g \to e^* \Omega_{A/R}^1$ une trivialisation du faisceau conormal relatif à la section unité. On définit l'action de $\mathcal{G}(R)$ sur $\omega$ par $g . \omega = \omega \circ g^{-1}$.
\end{defi}

Dans la suite, nous fixons un poids $\kappa \in X(T)$, avec $k_{1,i} \geq \dots \geq k_{g,i}$ pour tout $1 \leq i \leq d$.

\begin{defi}
Une forme modulaire de Siegel-Hilbert $F$ de niveau $\Gamma^0(p)$ sur $X$ de poids $\kappa \in X(T)$ est une loi qui à toute $O_K$-algèbre $R$, tout $R$-point $x=(A,\lambda,\iota,\eta, H_i)$ de $X$, et toute trivialisation $\omega : (R \otimes_{\mathbb{Z}_p} O_{F_p} )^g \to e^* \Omega_{A/R}^1$ associe un élément de $R$ noté $F(x, \omega)$, avec de plus les propriétés suivantes :
\begin{itemize}
\item Cette loi est fonctorielle en $R$.
\item Cette loi vérifie l'équation fonctionnelle suivante :
$$ \forall t \in T(R), u \in U(R), F (x,\omega \circ tu) = \kappa'(t) F(x, \omega)$$
\end{itemize}
\end{defi}

\begin{defi}
Soit $\mathcal{T} = $Isom$_{\mathcal{O}_X} ((\mathcal{O}_X \otimes_{\mathbb{Z}_p} O_{F_p} )^g , e^* \Omega_{A/X}^1)$, où $A$ est la variété universelle au-dessus de $X$. C'est un torseur sur $X$ sous le groupe $\mathcal{G} \times_{\mathbb{Z}_p} O_K$. Si $\kappa \in X(T)$ ; on définit $\omega^{\kappa} := \phi_* O_\mathcal{T}[\kappa']$, où $\phi : \mathcal{T} \to X$ est la projection, et $\phi_* O_\mathcal{T}[\kappa']$ désigne le sous-faisceau de $\phi_* O_\mathcal{T}$ où $B=T U$ agit par le caractère $\kappa'$ sur $T$ et le caractère trivial sur $U$. C'est le faisceau des formes modulaires de poids $\kappa$, et il est localement libre.
\end{defi}

\begin{prop}
Une forme modulaire de poids $\kappa$ à coefficients dans une $O_K$-algèbre $C$ est donc une section globale de $\omega^\kappa$, soit un élément de $H^0(X \times $Spec $C , \omega^\kappa)$.
\end{prop}

L'espace des formes de Hilbert-Siegel de poids $\kappa$ est $H^0(X, \omega^\kappa)$, et l'espace des formes modulaires sur $X_K$ est $H^0(X_K, \omega^\kappa)$. On note encore $\omega^\kappa$ l'analytifié de $\omega^\kappa$, qui est un faisceau sur $X_{rig}$. Faisons de plus l'hypothèse (inoffensive) que $gd > 1$ (le cas $g=d=1$ correspond au cas de la courbe modulaire, qui est bien connu). Cette hypothèse permet de négliger les pointes dans la définition des formes modulaires, à l'aide d'un principe de Koecher. Rappelons brièvement ce résultat.

\begin{theo}[$\cite{P-S 1}$ partie $6$]
Il existe une compactification toroïdale $\overline{X}$ de $X$ définie sur $O_K$, dépendant d'un choix combinatoire. Le schéma $\overline{X}$ est propre sur $O_K$, et le faisceau $\omega^\kappa$ s'étend à $\overline{X}$. De plus, on a le principe de Koecher algébrique et rigide, c'est-à-dire
\begin{displaymath}
\begin{array}{ccc}
H^0(\overline{X}, \omega^\kappa) = H^0 (X, \omega^\kappa) & \text{ et } & H^0(\overline{X}_{rig}, \omega^\kappa) = H^0 (X_{rig}, \omega^\kappa)
\end{array}
\end{displaymath}
où $\overline{X}_{rig}$ désigne l'espace rigide associé à $\overline{X}$.
\end{theo}

Puisque le schéma $\overline{X}$ est propre sur $O_K$, on a de plus par un principe GAGA ($\cite{EGA_3}$ partie $5.1$) $H^0(\overline{X} \times_{O_K} K, \omega^\kappa) = H^0(\overline{X}_{rig}, \omega^\kappa)$. En résumé, l'espace des formes modulaires est $H^0 (X_K,\omega^\kappa) = H^0 (X_{rig}, \omega^\kappa)$. Nous pouvons donc bien négliger les pointes dans la définition des formes modulaires. Définissons maintenant l'espace des formes surconvergentes.

\begin{defi}
L'espace des formes surconvergentes de poids $\kappa$ est le module 
$$H^0(X_K,\omega^\kappa)^\dagger := \text{colim}_\mathcal{V} H^0 (\mathcal{V},\omega^\kappa)$$
où la colimite est prise sur les voisinages stricts $\mathcal{V}$ du tube multiplicatif $X_{dg}$ dans $X_{rig}$.
\end{defi}

Une forme modulaire est donc en particulier une forme surconvergente. 

\begin{prop}
Nous avons donc une application injective
$$H^0(X,\omega^\kappa) \hookrightarrow H^0(X,\omega^\kappa)^\dagger$$
\end{prop}

\begin{proof}
Cela découle de l'unicité du prolongement analytique et du fait que le lieu ordinaire-multiplicatif rencontre toutes les composantes connexes de $X$. Justifions rapidement cette affirmation : les composantes connexes de $X$ sont les mêmes que celle de la variété sans niveau en $p$ (cela se démontre en passant aux complexes et en utilisant le demi-plan de Poincaré-Siegel). Il suffit ensuite d'utiliser le fait que le lieu ordinaire est dense sur $\mathbb{F}_p$ pour la variété sans niveau (voir $\cite{We}$). 
\end{proof}

\begin{defi}
L'image du morphisme précédent est l'espace des formes classiques.
\end{defi}

\subsection{Opérateur de Hecke} \label{defHecke}

Dans les sections précédentes, nous avons défini la correspondance de Hecke ensembliste $U_p$. Cette correspondance permet de définir un opérateur au niveau des formes modulaires. \\
Rappelons que l'on a défini l'espace de modules $C_K$ sur $K$ dont les $S$-points sont les $(A,\lambda,\iota,\eta, H_i,L)$ avec $(A,\lambda,\iota,\eta, H_i) \in X(S)$ et $L$ supplémentaire générique de $H_g$. \\
Nous avons deux morphismes $p_1, p_2 : C_{rig} \to X_{rig}$ : $p_1$ est l'oubli de $L$, et $p_2$ est le quotient par $L$. Notons également $\pi : A \to A/L$ l'isogénie universelle au-dessus de $C$.  \\
Celle-ci induit un isomorphisme $\pi^* : \omega_{(A/L)/X} \to \omega_{A/X}$, et donc un morphisme \\
$\pi^* (\kappa) :~p_2^* \omega^\kappa \to~p_1^* \omega^\kappa$. Pour tout ouvert $\mathcal{U}$ de $X_{rig}$, nous pouvons donc former le morphisme composé

\begin{displaymath}
\widetilde{U}_p :   H^0(U_p(\mathcal{U}),\omega^\kappa) \to H^0 ( p_1^{-1} (\mathcal{U}), p_2^* \omega^\kappa) \overset{\pi^*(\kappa)}{\to} H^0(p_1^{-1}(\mathcal{U}) , p_1^* \omega^\kappa) \overset{Tr_{p_1}}{\to}  H^0(\mathcal{U},\omega^\kappa)
\end{displaymath}

\begin{defi}
L'opérateur de Hecke agissant sur les formes modulaires est défini par $U_p :=~\frac{1}{p^{n_0}} \widetilde{U}_p$ avec $n_0=\frac{dg(g+1)}{2}$.
\end{defi}

\begin{rema}
Cette normalisation optimise l'intégrabilité de l'opérateur de Hecke, comme le montre un calcul sur les $q$-développements.
\end{rema}

Nous avons une description explicite de cet opérateur. Soit $\mathcal{U}$ un ouvert de $X_{rig}$, $f \in~H^0(U_p(\mathcal{U}),\omega^\kappa)$, $x=(A,\lambda,\iota,\eta, H_i) \in \mathcal{U}(\overline{K})$ et $\omega  \in \omega_{A/\overline{K}}$. Alors
$$(U_p f)(x,\omega) = \frac{1}{p^{n_0}} \sum_L f( A/L,\iota',\lambda',\eta', \text{Im}(H_i \to (A/L)[p]),\omega')$$
où $\omega' \in \omega_{(A/L)/\overline{K}}$ est définie par $\pi^* \omega' = \omega$, et la somme portant sur les supplémentaires génériques $L$ de $H_g$.  \\
Nous avons de plus des formules analogues pour les opérateurs $U_{p,i}^{good}$, $U_{p,i}^{bad}$ et $U_{p,i,j}^{bad}$, dans ce cas la somme portera seulement sur certains supplémentaires génériques $L$ de $H_g$ bien déterminés. \\
Ainsi, avec les notations de $\ref{zoneint}$ et $\ref{Hecke}$, on a 
$$(U_{p,i}^{good} f)(x,\omega) = \frac{1}{p^{n_0}} \sum_{\text{deg} L < 1 - \beta } f( A/L,\iota',\lambda',\eta', \text{Im}(H_i \to (A/L)[p]),\omega')$$
$$(U_{p,i}^{bad} f)(x,\omega) = \frac{1}{p^{n_0}} \sum_{\text{deg} L \geq 1 - \beta } f( A/L,\iota',\lambda',\eta', \text{Im}(H_i \to (A/L)[p]),\omega')$$
$$(U_{p,,\underline{m},i}^{bad} f)(x,\omega) = \frac{1}{p^{n_0}} \sum_{y \in U_p(x) \cap \mathcal{V}_i \backslash \mathcal{V}_{i+1} } f( y,\omega')$$

\begin{rema}
puisque $U_p$ stabilise les $X_{\geq dg - v}$, $v>0$, lesquels forment une base de voisinages du tube multiplicatif, l'opérateur de Hecke agit sur les formes modulaires surconvergentes.
\end{rema}

L'opérateur de Hecke $U_p$ agissant sur les formes modulaires, on peut le munir d'une norme d'opérateur.

\begin{defi}
Soit $\mathcal{U}$ un ouvert de $X_{rig}$, et $T$ un des opérateurs de Hecke définis précédemment agissant sur les formes modulaires $H^0 (T(\mathcal{U}), \omega^\kappa)$. On définit la norme de $T : H^0 (T(\mathcal{U}),\omega^\kappa) \to~H^0(\mathcal{U},\omega^\kappa)$ par
$$ \Vert T \Vert_\mathcal{U} := \inf \{r >0 , |Tf|_\mathcal{U} \leq r |f|_{T(\mathcal{U})}, \forall f \in H^0(T(\mathcal{U}),\omega^\kappa) \}$$
\end{defi}

\begin{prop} \label{norm}
Soit $\alpha$ et $\beta$ deux rationnels strictement compris entre $0$ et $1$, et $\mathcal{U} \subset X_{[0,dg-1+\alpha]}$ un ouvert sur lequel est défini un opérateur $U_p^{bad}$, cet opérateur étant défini pour le rationnel $\beta$. Alors
$$\Vert U_p^{bad} \Vert_\mathcal{U} \leq p^{n_0- (1 - \beta) \inf_i (k_{g,i})} $$
\end{prop}

L'opérateur $U_p^{bad}$ ne faisant intervenir que des supplémentaires génériques $L$ avec $\deg L~\geq~1-~\beta$, cette proposition découle du lemme suivant.

\begin{lemm} \label{lemnorm}
Soit $T$ un opérateur défini sur un ouvert $\mathcal{U}$, égal à $U_p$, $U_p^{good}$ ou $U_p^{bad}$. On suppose que cet opérateur ne fait intervenir que des supplémentaires génériques $L$ avec deg $L \geq c$, pour un certain $c \geq 0$. Alors
$$ \Vert T \Vert_\mathcal{U} \leq p^{n_0-c \inf_i k_{g,i}}$$
\end{lemm}

\begin{proof}
Soient $x$ un point de $X_{rig}$, correspondant à un couple $(A,\lambda,\iota,\eta, H_i)$ défini sur $O_{\overline{K}}$ et $L$ un supplémentaire générique de $H_g$. Alors le morphisme $\pi : A \to A/L$ donne une suite exacte de $O_{\overline{K}} \otimes_{\mathbb{Z}_p} O_{F_p}$-modules
$$0 \to \omega_{A/L} \overset{\pi^*}{\to} \omega_A \to \omega_L \to 0$$
En décomposant cette suite exacte selon les éléments de $Hom_{\mathbb{Q}_p}(F_p,\overline{K})$, on obtient
$$0 \to \omega_{A/L,i} \overset{\pi^*_i}{\to} \omega_{A,i} \to \omega_{L,i} \to 0$$
De plus, $v(\det \pi_i^*) = $deg$_i L$, où deg$_i L$ est le degré partiel défini dans la partie $\ref{partial}$. \\
On en déduit que 
$$\Vert \pi_i^* ( \kappa) \Vert \leq p^{- k_{g,i} \text{deg }_i L} $$
soit que 
$$\Vert \pi^*(\kappa) \Vert \leq p^{ - \sum_i k_{g,i} \text{deg }_i L} \leq p^{- \sum_i \inf_i k_{g,i} \text{deg }_i L} = p^{- \text{deg } L  \inf_i k_{g,i} }$$
Le résultat découle alors du fait que deg $L \geq c$.
\end{proof}

\subsection{Le théorème de prolongement analytique} \label{prolongement}

Dans ce paragraphe, nous allons démontrer le résultat suivant.

\begin{theo} \label{maintheo}
Soit $f$ une forme modulaire surconvergente de poids $\kappa = (k_{g,i} \leq~\dots \leq~k_{1,i})$, propre pour $U_p$ avec la valeur propre $a_p$, et supposons que 
$$v(a_p) + \frac{dg(g+1)}{2} < \inf_{1 \leq i \leq d} k_{g,i}$$
Alors $f$ est classique.
\end{theo}

\begin{rema}
le terme $\frac{dg(g+1)}{2}$ est le coefficient de normalisation $n$ de l'opérateur de Hecke, et est égal à la dimension de la variété de Siegel-Hilbert.
\end{rema}
 
\noindent	Soit donc $f$ une forme modulaire surconvergente propre pour $U_p$ de valeur propre $a_p$, avec $v(a_p) +~\frac{dg(g+1)}{2} < \inf k_{g,i}$. La forme modulaire $f$ est définie sur $X_{\geq dg - \gamma}$, pour un certain $\gamma >0$. On va prolonger $f$ à $X_{rig}$ tout entier. Par la dynamique de l'opérateur de Hecke (proposition $\ref{dyna}$), on peut prolonger $f$ à $X_{>dg-1}$ en utilisant l'équation fonctionnelle $f= a_p^{-1} U_p f$. \\
On va prolonger $f$ à $X_{[0, dg-1+\alpha]}$, pour un certain $\alpha >0$. Pour ce faire, nous allons utiliser la décomposition de la correspondance de l'opérateur $U_p^N$. \\
La condition sur le poids va nous garantir que les opérateurs $a_p^{-1} U_p^{bad}$ sont de normes strictement inférieure à $1$. En effet, cela résultera de la proposition $\ref{norm}$, et du fait que l'on peut prendre $\beta$ arbitrairement petit. Soit donc $\varepsilon$ tel que 
$$\frac{dg(g+1)}{2} + v(a_p) < (1-\varepsilon) \inf k_{g,i}$$
Fixons un rationnel $\alpha$ strictement positif avec $\alpha < \varepsilon$. Dans la partie $\ref{Hecke}$, nous avons obtenu une bonne suite d'ouverts de $\mathcal{U} = X_{[0, dg-1+ \alpha]}$ qui a permis de décomposer l'opérateur de Hecke. Nous allons maintenant utiliser cette décomposition. \\

Soient $N \geq 1$, $0<\beta_0 <1$ un rationnel et $(\mathcal{U}_i)_{i \in S_N}$ la bonne suite d'ouverts de $\mathcal{U}$ pour le rationnel $\beta$ obtenue dans la partie $\ref{Hecke}$. De plus, nous avons vu qu'il était possible de faire surconverger toute suite ainsi construite : soit $\beta^{(k)}$ une suite croissante de rationnels avec $\beta^{(0)}=\beta_0$. Pour tout $k \geq 1$, si $(\mathcal{U}_i^{(k)})$ est la bonne suite d'ouverts obtenus pour $\beta^{(k)}$ de $\mathcal{U}$, alors $\mathcal{U}_i^{(k)}$ est un voisinage strict de $\mathcal{U}_i^{(k-1)}$ dans $\mathcal{U}$, pour tout $i \in S_N$. On notera également $\mathcal{U}_i^{(0)} = \mathcal{U}_i$. On suppose également la suite $\beta^{(k)}$ bornée par $\varepsilon$. \\
Soient alors $\mathcal{V}_i = \mathcal{U}_{i}^{(i-1)}$ pour $i \geq 1$ et $\mathcal{V}_i' = \mathcal{U}_i^{(i)}$ pour $i \geq 0$. Alors $\mathcal{V}_i'$ est un voisinage strict de $\mathcal{V}_i$ pour tout $i \geq 1$. \\
Nous avons décomposé l'opérateur $U_p^N$ sur $\mathcal{V}_i' \backslash \mathcal{V}_{i+1}$ en 
$$U_{p}^N = \sum_{k=0}^{N-1} U_p^{N-1-k} T_k + T_N$$
avec $T_0 = U_{p,i}^{good}$, et pour $0 < k < N$
$$T_k = \sum_{i_1 \in S_{N-1}, \dots, i_k \in S_{N-k}}  U_{p,i_k}^{good} U_{p,i_{k-1},i_k}^{bad} \dots U_{p,i,i_1}^{bad}$$
et
$$T_N = \sum_{i_1 \in S_{N-1}, \dots, i_{N-1} \in S_{1}} U_{p,i_{N-1}}^{bad} U_{p,i_{N-2},i_{N-1}}^{bad} \dots U_{p,i,i_1}^{bad}$$  
Les images de $U_{p,i}^{good}$ et de $U_{p,i_k}^{good}$ ($i_k \in S_{N-k}$) sont incluses dans $X_{\geq dg -1 + \beta^{(i)}} \subset X_{\geq dg -1 + \beta_0}$, et les opérateur $U_{p,i,j}^{bad}$, $U_{p,i}^{bad}$ ne font intervenir que des supplémentaires $L$ de degré supérieur ou égal à $1 - \beta^{(i)} > 1 - \varepsilon$. 

\begin{defi}
Les séries de Kassaei sur $\mathcal{V}_i' \backslash \mathcal{V}_{i+1}$ sont définies par
$$f_{N,i} := a_p^{-1} U_{p,i}^{good} f + \sum_{k=1}^{N-1} \sum_{i_1 \in S_1, \dots, i_k \in S_k} a_p^{-k-1} U_{p,i,i_1}^{bad} \dots U_{p,i_{k-1},i_k}^{bad} U_{p,i_k}^{good} f$$
\end{defi}

Cette fonction est bien définie, puisque les opérateurs $U_{p,i}^{good}$ sont soit nuls, auquel cas leur action sur $f$ donne $0$, soit à valeur dans $X_{\geq dg - 1 + \beta_0}$ et $f$ est définie sur cet espace. Ce dernier espace étant quasi-compact, $f$ y est bornée, disons par $M$. \\
La proposition $\ref{norm}$ permet de majorer la norme des opérateurs $a_p^{-1} U_{p,i,j}^{bad}$ : la norme de ces opérateurs est inférieure à
$$u_0 = p^{n_0+v(a_p) - (1-\varepsilon) \inf_i k_{g,i}} < 1$$

\begin{lemm}
Les fonctions $f_{N,i}$ sont uniformément bornées.
\end{lemm}

\begin{proof}
On a 
$$ |a_p^{-k-1} U_{p,i,i_1}^{bad} \dots U_{p,i_{k-1},i_k}^{bad} U_{p,i_k}^{good} f |_{\mathcal{V}_i' \backslash \mathcal{V}_{i+1}} \leq u_0^k |a_p^{-1} U_{p,i_k}^{good} f |_{U_{p,i_{k-1},i_k}^{bad} \dots U_{p,i,i_1}^{bad} (\mathcal{V}_i' \backslash \mathcal{V}_{i+1} )} \leq |a_p^{-1} | p^{n_0}  M$$
car la norme de $U_{p,i_k}^{good}$ est majorée par $p^{n_0}$.
On peut donc majorer la fonction $f_{N,i}$ par
$$ |f_{N,i} |_{\mathcal{V}_i' \backslash \mathcal{V}_{i+1}} \leq |a_p^{-1} | p^{n_0} M$$
ce qui prouve que les fonctions $f_{N,i}$ sont uniformément bornées. 
\end{proof}

\begin{rema}
sur l'espace $\{ x, U_p^N (x) \cap \mathcal{U} = \emptyset \}$, on aurait pu définir $f_{N,i}$ par $a_p^{-N} U_p^N f$. Le fait de définir $f_{N,i}$ comme une série de Kassaei permet cependant de majorer plus facilement cette fonction. 
\end{rema} 

\begin{rema}
Cette démonstration est plus simple que celle utilisée dans d'autres articles (notamment $\cite{Ka}$ et $\cite{P-S 1}$).
\end{rema}

\noindent Nous avons donc défini une fonction $f_{N,i}$ sur $\mathcal{V}_i' \backslash \mathcal{V}_{i+1}$. Ces fonctions étant uniformément bornées, nous pouvons supposer qu'elles sont à valeurs entières, c'est-à-dire que ce sont des sections du faisceau $\widetilde{\omega}^\kappa$. Nous allons maintenant les recoller pour définir une fonction sur tout $\mathcal{U}$. 

\begin{lemm}
Soient $i,j \in S_N$ et $x \in (\mathcal{V}_i' \backslash \mathcal{V}_{i+1}) \cap (\mathcal{V}_j' \backslash \mathcal{V}_{j+1})$. Alors
$$ | (f_{N,i} -  f_{N,j}) (x) | \leq u_0^N M $$
\end{lemm}

\begin{proof}
Nous allons utiliser la décomposition effectuée dans le théorème $\ref{autre_decompo}$ pour obtenir une autre définition des séries de Kassaei. Cette décomposition pour le rationnel $\beta^{(i)}$ donne une bonne suite d'ouverts $(\mathcal{V}_k)_k$, et une décomposition de $U_{p^N}$ sur $\mathcal{V}_k \backslash \mathcal{V}_{k+1}$ pour tout $k$. Soit $k$ l'entier tel que $x \in \mathcal{V}_k \backslash \mathcal{V}_{k+1}$. Alors on a $f_{N,i}(x) = a_p^{-N} U_{p^N,k}^{good} f (x)$. De même, si 
$(\mathcal{V}_k')_k$ est la bonne suite d'ouverts obtenue dans le théorème $\ref{autre_decompo}$ pour $\beta^{(j)}$, et si $x \in \mathcal{V}_l' \backslash \mathcal{V}_{l+1}'$, alors $f_{N,j}(x) = a_p^{-N} U_{p^N,l}^{good} f (x)$. \\
Supposons par exemple que $i<j$. On a alors $\beta^{(i)} < \beta^{(j)}$, et au-dessus de $x$, l'op\'erateur $U_{p^N,l}^{bad}$ se d\'ecompose en ${U_{p^N,l}^{bad}}' + {U_{p^N,l}^{bad}}''$, l'op\'erateur ${U_{p^N,l}^{bad}}''$ ayant son image incluse dans $X_{]dg - 1 + \beta^{(i)} , dg - 1 + \beta^{(j)}]}$. On a alors $f_{N,i} (x) - f_{N,j} (x) = a_p^{-N} {U_{p^N,l}^{bad}}'' f(x)$. De plus, la norme de l'op\'erateur $a_p^{-N} {U_{p^N,l}^{bad}}''$ est inf\'erieure \`a $u_0^N$ d'apr\`es les calculs sur les normes des op\'erateurs de Hecke. D'o\`u
$$ | (f_{N,i} -  f_{N,j}) (x) | \leq u_0^N |f|_{{U_{p^N,l}^{bad}}''(x)} $$
De plus, l'ensemble ${U_{p^N,l}^{bad}}'' (x)$ \'etant inclus dans $X_{\geq d g -1 + \beta_0 }$, on a $|f|_{{U_{p^N,l}^{bad}}''(x)} \leq M$ ce qui donne la majoration.
\end{proof}

\begin{prop}
Il existe un entier $A_N$ telle que les fonctions $(f_{N,i})_{i\in S_N}$ se recollent en une fonction $g_N \in H^0(\mathcal{U}, \widetilde{\omega}^\kappa / p^{A_N})$.
\end{prop}

\begin{proof}
La décomposition de l'ouvert $\mathcal{U}$ étant finie, soit $L$ tel que $\mathcal{V}_{L+1}$ soit vide. La fonction $f_{N,L}$ est donc définie sur $\mathcal{V}_L'$. La fonction $f_{N,L-1}$ est elle définie sur $\mathcal{V}_{L-1}' \backslash \mathcal{V}_L$. \\
De plus, d'après le lemme précédent, on a
$$ | f_{N,L-1} - f_{N,L} |_{(\mathcal{V}_L' \cap \mathcal{V}_{L-1}') \backslash \mathcal{V}_L} \leq u_0^N M$$
Soit $A_N$ tel que $u_0^N M \leq p^{-A_N}$ ; comme $u_0 < 1$, la suite $(A_N)$ tend vers l'infini. \\
Les fonctions $f_{N,L-1}$ et $f_{N,L}$ sont donc égales modulo $p^{A_N}$ sur $(\mathcal{V}_L' \cap \mathcal{V}_{L-1}') \backslash \mathcal{V}_L$. Comme $(\mathcal{V}_L' \cap~\mathcal{V}_{L-1}' , \mathcal{V}_{L-1}' \backslash  \mathcal{V}_L)$ est un recouvrement admissible de $\mathcal{V}_{L-1}'$ , celles-ci se recollent en une fonction $g_{N,L-1} \in H^0 ( \mathcal{V}_{L-1}', \widetilde{\omega}^\kappa / p^{A_N})$. \\
De même, $g_{N,L-1}$ et $f_{N,L-2}$ sont égales (modulo $p^{A_N}$) sur $(\mathcal{V}_{L-2}' \cap \mathcal{V}_{L-1} ') \backslash \mathcal{V}_{L-1}$, et donc se recollent en $g_{N,L-2} \in H^0 (\mathcal{V}_{L-2}' , \widetilde{\omega}^\kappa / p^{A_N})$. \\
En répétant ce processus, on voit que les fonctions $f_{N,i}$ se recollent toutes modulo $p^{A_N}$ sur $\mathcal{V}_0' = \mathcal{U}$, et définissent donc une fonction $g_N \in H^0(\mathcal{U}, \widetilde{\omega}^\kappa / p^{A_N})$.

\end{proof}

\begin{prop}
Les fonctions $(g_N)$ définissent un système projectif dans $\underset{\leftarrow}\lim $ $ H^0 ( \mathcal{U}, \widetilde{\omega}^\kappa / p^m)$.
\end{prop}

\begin{proof}
Nous allons prouver que $g_{N+1}$ et $g_N$ sont égales modulo $p^{A_N}$. Soit $ x \in \mathcal{U}$ ; nous avons construit en $x$ les séries de Kassaei $f_{N,i}$ et $f_{N+1,k}$. Or le terme $f_{N+1,k}$ provient d'une décomposition de $U_p^{N+1}$ du type
$$U_p^{N+1} = \sum_{l=0}^{N} U_p^{N-l} T_N + T_{N+1}$$
Nous pouvons donc écrire $f_{N+1,k} = h_1 + h_2$, la fonction $h_1$ étant associée à l'opérateur $\sum_{l=0}^{N-1} U_p^{N-1-l} T_N$  
et $h_2$ à $T_N$. \\
Or la fonction $h_1$ est en réalité une série de Kassaei pour une certaine décomposition de $U_p^N$ : le lemme précédent donne donc 
$$ | (f_{N,i} - h_1) (x) | \leq p^{-A_N}$$
De plus, on a
$$h_2 = \sum_{i_1 \in S_1, \dots, i_N \in S_N} a_p^{-N-1} U_{p,i,i_1}^{bad} \dots U_{p,i_{N-1},i_N}^{bad} U_{p,i_N}^{good} f $$
donc comme les opérateurs $a_{p}^{-1} U_{p,i,j}^{bad}$ ont une norme inférieure à $u_0$, 
$$ | h_2(x) | \leq u_0^N p^{n_0} |a_p^{-1} | M = p^{-A_N'}$$
avec $A_N'=A_N - n_0 - v(a_p)$. Quitte à remplacer $A_N$ par $A_N'$, on voit donc que la réduction de $g_{N+1}$ modulo $p^{A_N}$ est égal à $g_N$.
\end{proof}

En utilisant le gluing lemma (lemme $\ref{glue}$), on voit donc que les fonctions $g_N$ définissent une fonction $g \in H^0(\mathcal{U},\omega^\kappa)$. Bien sûr, $g$ se recolle avec $f$ sur $X_{>dg - 1}$. En effet, si $x \in X_{>dg - 1}$, il existe $N_0$ tel que $U_p^{N} (x) \subset X_{\geq dg-1+\varepsilon}$ pour $N \geq N_0$, et la série de Kassaei est alors stationnaire égale à 
$$ a_p^{-N_0} U_p^{N_0} f$$

Nous avons donc étendu $f$ sur $X_{[0, dg-1+\alpha]}$, pour un certain $\alpha > 0$. Comme le recouvrement $(X_{[0,dg-1+\alpha]}, X_{>dg-1})$ est un recouvrement admissible de $X_{rig}$, on peut donc prolonger $f$ à tout $X_{rig}$, ce qui prouve que $f$ est classique.

\section{Cas des variétés de Shimura de type~(C)} \label{symplectique}

Le théorème que nous avons démontré se généralise au cas d'une variété de Shimura PEL de type (C). 

\subsection{Données de Shimura}

Rappelons les données paramétrant les variétés de Shimura PEL de type (C) (voir $\cite{Ko}$). Soit $B$ une $\mathbb{Q}$-algèbre simple munie d'une involution positive $\star$. Soit $F$ le centre de $B$ et $F_0$ le sous-corps de $F$ fixé par $\star$. Le corps $F_0$ est une extension totalement réelle de $\mathbb{Q}$, soit $d$ son degré. Faisons les hypothèses suivantes :

\begin{itemize}
\item $F=F_0$.
\item Pour tout plongement $F \to \mathbb{R}$, $B \otimes_{F} \mathbb{R} \simeq $M$_n(\mathbb{R})$, et l'involution $\star$ est donnée par $A \to A^t$.
\end{itemize}

Soit également $(U_{\mathbb{Q}},\langle,\rangle)$ un $B$-module hermitien non dégénéré. Soit $G$ le groupe des automorphismes du $B$-module hermitien $U_{\mathbb{Q}}$ ; pour toute $\mathbb{Q}$-algèbre $R$, on a donc
$$G(R) = \left\{ (g,c) \in GL_{B} (U_{\mathbb{Q}} \otimes_{\mathbb{Q}} R) \times R^* , \langle gx,gy \rangle =c \langle x,y\rangle \text{ pour tout } x,y \in U_{\mathbb{Q}} \otimes_{\mathbb{Q}} R \right\} $$

Soient $\tau_1, \dots, \tau_d$ les plongements de $F$ dans $\mathbb{R}$,et $B_i=B \otimes_{F,\tau_i} \mathbb{R} \simeq $M$_n (\mathbb{R})$. Alors $G_\mathbb{R}$ est isomorphe à
$$\text{G} \left( \prod_{i=1}^d \text{Sp}_{2 g} \right)$$
où $g = \frac{1}{2nd} $dim$_\mathbb{Q} U_\mathbb{Q}$. \\
Donnons-nous également un morphisme de $\mathbb{R}$-algèbres $h : \mathbb{C} \to $End$_B U_\mathbb{R}$ tel que $\langle h(z)v,w\rangle ~=~\langle v,h(\overline{z})w\rangle$ et $(v,w) \to \langle v,h(i)w\rangle$ est définie positive. Ce morphisme définit donc une structure complexe sur $U_\mathbb{R}$ : soit $U^{1,0}_{\mathbb{R}}$ le sous-espace de $U_\mathbb{R}$ pour lequel $h(z)$ agit par la multiplication par $z$. \\
On a alors $U^{1,0}_{\mathbb{R}} \simeq \prod_{i=1}^d (\mathbb{R}^n)^g$ en tant que $B \otimes_{\mathbb{Q}} \mathbb{R} \simeq \oplus_{i=1}^d $M$_n( \mathbb{R})$-module. \\
Soient également un ordre $O_B$ de $B$ stable par par $\star$, et un réseau $U$ de $U_\mathbb{Q}$ tel que l'accouplement $\langle,\rangle$ restreint à $U\times U$ soit à valeurs dans $\mathbb{Z}$. Nous ferons également les hypothèses suivantes : 
\begin{itemize}
\item $B \otimes_{\mathbb{Q}} \mathbb{Q}_p$ est isomorphe à un produit d'algèbres de matrices à coefficients dans une extension non ramifiée de $\mathbb{Q}_p$.
\item $O_B$ est un ordre maximal en $p$.
\item L'accouplement $U \times U \to \mathbb{Z}$ est parfait en $p$.
\end{itemize}
$ $\\
Notons $\mathbb{Z}_{(p)}$ le localisé de $\mathbb{Z}$ en $p$ ; $O_B$ est un $\mathbb{Z}_{(p)}$-module libre. Soit $e_1, \dots, e_t$ une base de ce module, et 
$$\text{det}_{U^{1,0}} = f(X_1, \dots, X_t) = \det (X_1 \alpha_1 + \dots + X_t \alpha_t ;U^{1,0}_{\mathbb{C}} \otimes_{\mathbb{C}} \mathbb{C} [X_1, \dots, X_t])$$
On montre ($\cite{Ko}$) que $f$ est un polynôme à coefficients algébriques. Le corps de nombres $E$ engendré par ses coefficients est appelé corps réflexe, et est égal à $\mathbb{Q}$ dans le cas $(C)$. \\
De plus, d'après les hypothèses précédentes, $p$ est non ramifié dans $F$. Soit $h$ le nombre d'idéaux premiers de $F$ au-dessus de $p$, et $d_i$ le degré résiduel de chacune de ces places. Alors $O_B \otimes~\mathbb{Z}_p \simeq~\prod_{i=1}^{h} $M$_n(\mathbb{Z}_{p^{d_i}})$, où $\mathbb{Z}_{p^{d_i}}$ est l'anneau des entiers de l'unique extension non ramifiée de degré $d_i$ de $\mathbb{Q}_p$.

\subsection{Variété de Shimura} \label{defshi}

Définissons maintenant la variété de Shimura PEL de type (C) associée à $G$. Soit $N \geq 3$ un entier premier à $p$, et $K$ une extension finie de $\mathbb{Q}_p$ contenant tous les plongements $F \to \overline{\mathbb{Q}_p}$. On notera $O_K$ l'anneau des entiers de $K$. 

\begin{defi}
Soit $X$ sur Spec $O_K$ l'espace de modules dont les $S$-points sont les classes d'isomorphismes des $(A,\lambda,\iota,\eta)$ où
\begin{itemize}
\item $A \to S$ est un schéma abélien
\item $\lambda : A \to A^t$ est une polarisation de degré premier à $p$.
\item $\iota : O_B \to $End $A$ est compatible avec les involutions $\star$ et de Rosati, et les polynômes $\det_{U^{1,0}}$ et $\det_{Lie (A)}$ sont égaux.
\item $\eta : A[N] \to U/NU$ est une similitude symplectique $O_B$-linéaire, qui se relève localement pour la topologie étale en une similitude symplectique $O_B$-linéaire
$$H_1 (A,\mathbb{A}_f^p) \to U \otimes_{\mathbb{Z}} \mathbb{A}_f^p$$
\end{itemize}
\end{defi}

\begin{prop}
L'espace $X$ est un schéma quasi-projectif sur $O_K$.
\end{prop}

La condition du déterminant est explicite ; notons $St$ est le $O_B \otimes_{\mathbb{Z}} \mathbb{Z}_p$-module défini par
$$St = \bigoplus_{i=1}^h (\mathbb{Z}_{p^{d_i}}^n)^g$$
où l'action de $O_B \otimes~\mathbb{Z}_p \simeq~\prod_{i=1}^{h} $M$_n(\mathbb{Z}_{p^{d_i}})$ est l'action standard sur chacun des facteurs. Alors la condition du déterminant est équivalente au fait que Lie$(A)$ soit isomorphe localement pour la topologie de Zariski à $St \otimes_{\mathbb{Z}_p} \mathcal{O}_S$ comme $O_B \otimes_{\mathbb{Z}} \mathcal{O}_S$-module. On voit donc en particulier que le schéma abélien est de dimension $ndg$.\\
Nous allons maintenant définir une structure de niveau Iwahorique sur $X$. \\
Soit $\pi_1, \dots, \pi_{h}$ les idéaux premiers de $F$ au-dessus de $p$. Si $A \to S$ est un schéma abélien, on a donc
$$A[p^\infty] = \oplus_{i=1}^{h} A[\pi_i^\infty]$$
De plus, les groupes de Barsotti-Tate $A[\pi_i^\infty]$ sont principalement polarisés de dimension $nd_i g$, et muni d'une action M$_n(\mathbb{Z}_{p^{d_i}})$. 

\begin{defi}
Soit $X_{Iw}$ l'espace de modules sur $\mathbb{Z}_p$ dont les $S$-points sont les $(A,\lambda,\iota,\eta,H_{i,j})$ où $(A,\lambda,\iota,\eta) \in X(S)$ et $0=H_{i,0} \subset H_{i,1} \subset \dots \subset H_{i,g}$ est un drapeau de sous-groupes finis et plats, stables par $O_B$, et totalement isotropes de $A[\pi_i]$, chaque $H_{i,j}$ étant de hauteur $nd_i j$, pour tout $1\leq i \leq h$.
\end{defi}

Nous noterons $X_{rig}$ et $X_{Iw,rig}$ les espaces rigides associés respectivement à $X$ et $X_{Iw}$.

\subsection{Formes modulaires}

Soit $A$ le schéma abélien universel sur $X$, et soit $e^* \Omega_{A/X}^1$ le faisceau conormal relatif à la section unité de $A$ ; il est localement pour la topologie de Zariski isomorphe à $St \otimes_{\mathbb{Z}_p} \mathcal{O}_X$ comme $O_B \otimes_{\mathbb{Z}} \mathcal{O}_X$-module, où on rappelle que
$$St = \oplus_{i=1}^{h} (\mathbb{Z}_{p^{d_i}}^n )^g $$
Soit $\mathcal{T} = $Isom$_{O_B \otimes \mathcal{O}_X} (St \otimes_{\mathbb{Z}_p} \mathcal{O}_X, e^* \Omega_{A/X}^1)$. C'est un torseur sur $X$ sous le groupe
$$M=\left ( \prod_{i=1}^{h} Res_{\mathbb{Z}_{p^{d_i}} / \mathbb{Z}_p} GL_g \right ) \times_{\mathbb{Z}_p} O_K$$

Soit $B_M$ le Borel supérieur de $M$, $U_M$ son radical unipotent, et $T_M$ son tore maximal. Soit $X(T_M)$ le groupe des caractères de $T_M$, et $X(T_M)^+$ le cône des poids dominants pour $B_M$. Si $\kappa \in X(T_M)^+$, on note $\kappa'=- w_0 \kappa \in X(T_M)^+$, où $w_0$ est l'élément le plus long du groupe de Weyl de $M$ relativement à $T_M$. \\
Soit $\phi : \mathcal{T} \to X$ le morphisme de projection.

\begin{defi}
Soit $\kappa \in X(T_M)^+$. Le faisceau des formes modulaires de poids $\kappa$ est $\omega^\kappa =~\phi_* O_\mathcal{T}[\kappa']$, où $\phi_* O_\mathcal{T}[\kappa']$ est le sous-faisceau de $\phi_* O_\mathcal{T}$ où $B_M=T_M U_M$ agit par $\kappa$ sur $T_M$ et trivialement sur $U_M$.
\end{defi} 

Une forme modulaire de poids $\kappa$ sur $X$ est donc une section globale de $\omega^\kappa$, soit un élément de $H^0(X , \omega^\kappa)$. En utilisant la projection $X_{Iw} \to X$, on définit de même le faisceau $\omega^\kappa$ sur $X_{Iw}$, ainsi que les formes modulaires sur $X_{Iw}$. 

\begin{rema}
Par équivalence de Morita, la catégorie des $M_n(A)$-modules et celle des $A$-modules sont équivalentes, pour tout anneau $A$. L'équivalence de catégorie est explicite : à un $A$-module $M$, on associe $M^n$, qui est bien muni d'une action de $M_n(A)$ ; réciproquement, à un $M_n(A)$-module $N$, on associe le $A$-module $E_{1,1} N$, où $E_{1,1}$ est la matrice avec un seul coefficient non nul en position $(1,1)$ égal à $1$. \\
De cette manière, puisque $O_B \otimes_{\mathbb{Z}} \mathbb{Z}_p = \prod_{i=1}^h M_n(\mathbb{Z}_{p^{d_i}})$, et que le faisceau $\omega_A$ est isomorphe à $St \otimes \mathcal{O}_{X}$ comme $O_B \otimes_{\mathbb{Z}} \mathcal{O}_{X}$-module, l'équivalence de Morita associe à $\omega_A$ le faisceau de $(\prod_{i=1}^h \mathbb{Z}_{p^{d_i}}) \otimes_{\mathbb{Z}} \mathcal{O}_{X}$-modules défini par $E \cdot \omega_A$, où $E$ est la projection défini par $(E_{1,1})_i \in \prod_{i=1}^h M_n(\mathbb{Z}_{p^{d_i}})$. Ce faisceau est isomorphe à $(\oplus_{i=1}^{h} \mathbb{Z}_{p^{d_i}}^g) \otimes_{\mathbb{Z}} \mathcal{O}_{X}$ comme $(\prod_{i=1}^h \mathbb{Z}_{p^{d_i}}) \otimes_{\mathbb{Z}} \mathcal{O}_{X}$-module.
\end{rema}

En utilisant la remarque précédente, on voit qu'une définition équivalente du torseur $\mathcal{T}$ est
$$ \mathcal{T} = \prod_{i=1}^h \text{Isom}_{\mathbb{Z}_{p^{d_i}} \otimes \mathcal{O}_X} (\mathbb{Z}_{p^{d_i}}^g \otimes_{\mathbb{Z}_p} \mathcal{O}_X, E_i \cdot e^* \Omega_{A/X}^1)  $$
où $E_i$ est l'élément de $\prod_{j=1}^h M_n(\mathbb{Z}_{p^{d_j}})$ dont toutes les coordonnées sont nulles, sauf la $i$-ième égal à la matrice $E_{1,1}$. Nous utiliserons très souvent l'équivalence de Morita, ce qui nous permettra de se ramener au cas où $n=1$ (i.e. au cas où $B$ est simple). 

\begin{rema}
Le poids $\kappa$ d'une forme modulaire est une famille d'entiers
$$\prod_{i=1}^{h}  \prod_{j=1}^{d_i} (k_{1,j,i} \geq \dots \geq k_{g,j,i} ) $$
\end{rema}

Comme dans le cas précédent, nous n'avons pas à nous préoccuper des pointes pour la définition des formes modulaires, car il existe des modèles entiers des compactifications toroïdales.

\begin{theo}[$\cite{P-S 1}$ partie $6$] \label{toro}
Il existe une compactification toroïdale $\overline{X}_{Iw}$ de $X_{Iw}$ définie sur $O_K$, dépendant d'un choix combinatoire. Le schéma $\overline{X}_{Iw}$ est propre sur $O_K$, et le faisceau $\omega^\kappa$ s'étend à $\overline{X}_{Iw}$. De plus, on a le principe de Koecher algébrique et rigide, c'est-à-dire $H^0(\overline{X}_{Iw}, \omega^\kappa) = H^0 (X_{Iw}, \omega^\kappa)$ et $H^0(\overline{X}_{Iw,rig}, \omega^\kappa) = H^0 (X_{Iw,rig}, \omega^\kappa)$ où $\overline{X}_{Iw,rig}$ désigne l'espace rigide associé à $\overline{X}_{Iw}$.
\end{theo}

Définissons maintenant les fonctions degrés.

\begin{defi}
Soit $i$ un entier entre $1$ et $h$. On définit la fonction $Deg_i : X_{Iw,rig} \to~[0,d_i g]$ par $Deg_i((A,\lambda,\iota,\eta,H_{j,k})) = 1/n \deg H_{i,g}$. On définit également la fonction\\
$Deg : X_{Iw,rig} \to~\prod_{i=1}^h [0,d_i g]$ par $x \to (Deg_i(x))$.
\end{defi}

\begin{rema}
Le fait de diviser par $n$ est lié également à l'action de l'algèbre de matrices. En effet, le schéma en groupe $H_{i,g}$ est muni d'une action de M$_n(\mathbb{Z}_{p^{d_i}})$, et donc est isomorphe à $n$ copies de $E_{1,1} H_{i,g}$. La quantité pertinente à étudier n'est donc pas le degré de $H_{i,g}$, mais celui de $E_{1,1} H_{i,g}$ qui est de hauteur $d_i g$. C'est pourquoi nous avons défini la fonction $Deg_i$ comme le degré de $H_{i,g}$ divisé par $n$, qui est égal au degré de $E_{1,1} H_{i,g}$.
\end{rema}
$ $\\
Soit $X_{Iw,rig}^{mult}$ le lieu ordinaire-multiplicatif de $X_{Iw,rig}$ ; il est égal par définition à \\
$Deg^{-1} (\{d_1 g \} \times~\dots \times~\{d_h g \})$. Nous pouvons maintenant définir les formes surconvergentes sur pour $X$.

\begin{defi}
L'ensemble des formes modulaires surconvergentes est défini par
$$H^0(X_{Iw,rig},\omega^\kappa)^\dagger := \text{colim}_\mathcal{V} H^0 (\mathcal{V},\omega^\kappa)$$
où la colimite est prise sur les voisinages stricts $\mathcal{V}$ de $X_{Iw,rig}^{mult}$ dans $X_{Iw,rig}$.
\end{defi}

\subsection{Opérateurs de Hecke}

Soit $1 \leq i \leq h$. Soit $C_i$ l'espaces des modules sur $K$ dont les $S$-points sont les $(A,\lambda,\iota,\eta,H_{j,k},L)$ avec $(A,\lambda,\iota,\eta,H_{j,k}) \in X_{Iw}(S)$ et $L$ un sous-groupe fini et plat de $A[\pi_i]$, stable par $O_B$, totalement isotrope et supplémentaire de $H_{i,g}$ dans $A[\pi_i]$. \\
Nous avons deux morphismes $p_1, p_2 : C_i \to X_{Iw,K}=X_{Iw} \times K$ : $p_1$ est l'oubli de $L$, et $p_2$ est le quotient par $L$. Il y a a priori une ambiguïté pour définir la projection $p_2$, car il faut préciser quelle polarisation prendre pour le schéma abélien $A/L$. Fixons un élément $x_i$ totalement positif de $O_{F}$, de valuation $\pi_i$-adique $1$, et de valuation $\pi_j$-adique $0$ si $j \neq i$. Alors on définit la polarisation $\lambda'$ sur $A/L$ comme la polarisation descendue $x_i \cdot \lambda$. C'est bien une polarisation de degré premier à $p$. Remarquons que dans le cas où il n'y a qu'une seule place au-dessus de $p$, on peut prendre $x_i = p$ (c'est ce qui était fait dans les paragraphes précédents). Cela prouve que l'opérateur $U_p$ est défini canoniquement, mais qu'en général, les opérateurs géométriques $U_{p,i}$, associés à la place $\pi_i$, que nous allons définir ne sont pas canoniques. \\
Soit $C_i^{an}$ l'analytifié de $C_i$, et $C_{i,rig} = p_1^{-1} (X_{Iw,rig})$. Nous noterons encore $p_1,p_2$ les morphismes induits $C_{i,rig} \to X_{Iw,rig}$. L'opérateur de Hecke géométrique est défini par $U_{p,i} (S) = p_2(p_1^{-1}(S))$ pour toute partie $S$ de $X_{Iw,rig}$. \\
Notons $\pi : A \to A/L$ l'isogénie universelle au-dessus de $C_i$. Celle-ci induit un isomorphisme $\pi^* : \omega_{(A/L)/X} \to \omega_{A/X}$, et donc un morphisme $\pi^* (\kappa) : p_2^* \omega^\kappa \to~p_1^* \omega^\kappa$. Pour tout ouvert $\mathcal{U}$ de $X_{Iw,rig}$, nous pouvons donc former le morphisme composé
\begin{displaymath}
\widetilde{U}_{p,i} :   H^0(U_p(\mathcal{U}),\omega^\kappa) \to H^0 ( p_1^{-1} (\mathcal{U}), p_2^* \omega^\kappa) \overset{\pi^*(\kappa)}{\to} H^0(p_1^{-1}(\mathcal{U}) , p_1^* \omega^\kappa) \overset{Tr_{p_1}}{\to}  H^0(\mathcal{U},\omega^\kappa)
\end{displaymath}

\begin{defi}
L'opérateur de Hecke agissant sur les formes modulaires est alors défini par $U_{p,i} = \frac{1}{p^{n_i}} \widetilde{U}_{p,i}$ avec $n_i=\frac{d_i g(g+1)}{2}$. 
\end{defi}

\begin{rema} \label{non_canonique}
Puisque la définition de la projection $p_2$ nécessite le choix d'un élément $x_i$, l'opérateur de Hecke n'est pas défini canoniquement. Il est cependant canonique pour l'action sur les formes modulaires, si on se restreint aux formes invariantes pour l'action d'un certain groupe. Puisque ce problème ne crée pas de difficulté pour nous, nous n'entrons pas dans les détails. Remarquons enfin que les opérateurs $U_{p,i}$ commutent entre eux. En effet, l'opérateur $U_{p,i} U_{p,j}$ est obtenu à partir de l'espace des $(A,\lambda,\iota,\eta,H_{j,k},L)$, où $L=L_1 \oplus L_2$, avec $L_1 \subset A[\pi_i]$ et $L_2 \subset A[\pi_j]$ comme précédemment. La polarisation sur le schéma abélien $A/L$ est alors la polarisation descendue $(x_i x_j) \cdot \lambda$.
\end{rema}

L'opérateur $U_{p,i}$ se comporte bien avec les fonctions degrés.

\begin{prop}
Soit $x \in X_{Iw,rig}$, et $y \in U_{p,i}(x)$. Alors
\begin{itemize}
\item Si $j \neq i$, alors $Deg_j(y) = Deg_j(x)$.
\item $Deg_i (y) \geq Deg_i(x)$.
\item Si $Deg_i(x) = Deg_i(y)$, alors $Deg_i(x) \in \mathbb{Z}$.
\end{itemize}
\end{prop}

\begin{proof}
Supposons que $x$ corresponde à un couple $(A,\lambda,\iota,\eta,H_{i,j})$ défini sur une extension $K_1$ de $K$, et que $y$ corresponde à un sous-groupe $L$ de $A[\pi_i]$. Alors, pour tout $j$, $Deg_j(y) = 1/n \deg(H_{j,g}')$ où $H_{j,g}'$ est l'image de $H_{j,g}$ dans $A/L$. Si $j \neq i$, alors $H_{j,g}$ et $L$ sont en somme directe dans $A[p]$, ce qui implique que $Deg_j(y) = Deg_j(x)$. Le deuxième point résulte du fait que le morphisme $H_{i,g} \to H_{i,g}'$ est un isomorphisme en fibre générique et des propriétés de la fonction degré. \\
Supposons maintenant que $Deg_i(x) = Deg_i(y)$. Alors, d'après les propriétés de la fonction degré, cela implique que $H_{i,g}$ et $L$ sont en somme directe dans $A[\pi_i]$, soit que $A[\pi_i] =~H_{i,g} \oplus~L$. En appliquant le projecteur $E_{1,1}$, on obtient que $E_{1,1} A[\pi_i] = E_{1,1}H_{i,g} \oplus~E_{1,1} L$. Or $E_{1,1} A [\pi_i]$ est un Barsotti-Tate tronqué d'échelon $1$. Cela implique que $E_{1,1} H_{i,g}$ et $E_{1,1} L$ le sont également, et en particulier leurs degrés sont entiers.
\end{proof}

L'opérateur $U_{p,i}$ augmente donc la $i$-ième fonction degré, et laisse les autres inchangées. De plus, il augmente strictement la fonction $Deg_i$, sauf aux points de degrés entiers. Nous avons comme précédemment un résultat de contraction.

\begin{prop} \label{augmente_C}
Soit $r$ un entier compris entre $0$ et $d_ig -1$. Soit $0 < \lambda < \mu < 1$ des réels. Alors, il existe un entier $N$ tel que
$$U_{p,i}^N (Deg_i^{-1}([r + \lambda,d_i g])) \subset Deg_i^{-1}([r + \mu,d_i g])$$
\end{prop}

Nous avons donc défini $h$ opérateurs agissant sur les formes modulaires. Remarquons également que puisque ces opérateurs augmentent le degré, ils agissent sur les formes surconvergentes. \\
Enfin, nous pouvons également décomposer les opérateurs de Hecke suivant le nombre de bons ou mauvais supplémentaires. Pour simplifier les notations, nous énonçons le théorème pour $U_{p,1}$. Soit $\alpha$ un rationnel avec $0< \alpha <1$, $I_2, \dots, I_h$ des intervalles compacts avec $I_k \subset [0,d_k g]$. On pose $\mathcal{U} = Deg^{-1} ([0,d_1g - 1 +\alpha] \times I_2 \times \dots \times I_h)$.

\begin{theo}  \label{bigtheo_C}
Soit $N \geq 1$ et $\beta$ un rationnel avec $0<\beta<1$. Il existe une bonne suite d'ouverts $(\mathcal{U}_i (N))_{i \in S_N}$ de $\mathcal{U}$ , tels que pour tout $i\geq 0$, on peut décomposer la correspondance $U_{p,1}^N$ sur $\mathcal{U}_{i}(N) \backslash \mathcal{U}_{i+1}(N)$ en 
$$ U_{p,1}^N = \left ( \coprod_{k=0}^{N-1} U_{p,1}^{N-1-k} \circ T_k  \right ) \coprod T_N$$
avec $T_0 = U_{p,i,N}^{good}$, pour $0 < k < N$
$$T_k = \coprod_{i_1 \in S_{N-1}, \dots, i_k \in S_{N-k}}  U_{p,1,i_k,N}^{good} U_{p,1,i_{k-1},i_k,N}^{bad} \dots U_{p,1,i,i_1,N}^{bad}$$
et
$$T_N = \coprod_{i_1 \in S_{N-1}, \dots, i_{N-1} \in S_1} U_{p,1,i_{N-1},N}^{bad} U_{p,1,i_{N-2},i_{N-1},N}^{bad} \dots U_{p,1,i,i_1,N}^{bad}$$  
avec
\begin{itemize}
\item les images des opérateurs $U_{p,1,j,N}^{good}$ ($j \in S_k$) sont incluses dans $Deg_1^{-1} ( ] d_1 g - 1 + \beta, d_1 g])$
\item les opérateurs $U_{p,1,i,j,N}^{bad}$ ($i \in S_k$, $j \in S_{k-1})$ et $U_{p,j,N}^{bad}$ ($j \in S_1$) sont obtenus en quotientant par un sous-groupe $L$ de degré supérieur ou égal à $n(1 - \beta)$, et ont donc leurs images incluses dans $Deg_1^{-1}([0, d_1g-1 + \beta])$.
\end{itemize}
Enfin, si $\beta'$ est un autre rationnel avec $\beta < \beta' <1$, et si $(\mathcal{U}_i'(N))$ est la bonne suite d'ouverts obtenue pour $\beta'$, alors $\mathcal{U}_i'(N)$ est un voisinage strict de $\mathcal{U}_i(N)$ pour tout $i$.
\end{theo}

Un point clé dans l'étape du prolongement analytique est la majoration de la norme des opérateurs $U_{p,1}^{bad}$. 

\begin{prop} \label{norm_C}
Soit $T$ un opérateur défini sur un ouvert $\mathcal{U}$, égal à $U_{p,1}$, $U_{p,1}^{good}$ ou $U_{p,1}^{bad}$. On suppose que cet opérateur ne fait intervenir que des supplémentaires génériques $L$ de $H_{g,1}$ avec deg $L \geq nc$, pour un certain $c \geq 0$. Alors
$$ \Vert T \Vert_\mathcal{U} \leq p^{n_1-c \inf_j k_{g,j,1}}$$
\end{prop}

\begin{proof}
Soient $x$ un point de $X_{Iw,rig}$, correspondant à un couple $(A,\lambda,\iota,\eta, H_{i,j})$ défini sur $O_{\overline{K}}$ et $L$ un supplémentaire générique de $H_{1,g}$. Alors le morphisme $\pi : A \to A/L$ donne une suite exacte de $O_{\overline{K}} \otimes_{\mathbb{Z}} O_{B}$-modules
$$0 \to \omega_{A/L} \overset{\pi^*}{\to} \omega_A \to \omega_L \to 0$$
On rappelle que $O_B \otimes_{\mathbb{Z}} \mathbb{Z}_p = \prod_{i=1}^{h} $M$_n(\mathbb{Z}_{p^{d_i}})$, et on note $E_i = (E_{1,1})_i$ l'élément dont toutes les coordonnées sont nulles, sauf la $i$-ième égale à la matrice $E_{1,1}$. En appliquant le projecteur $E_i$ à la suite précédente on obtient
$$0 \to E_i \omega_{A/L} \overset{\pi_i^*}{\to} E_i \omega_A \to E_i \omega_L \to 0$$
Puisque $L \subset A[\pi_1]$, $\pi_i^*$ est un isomorphisme si $i \neq 1$. De plus, 
$$v(\det \pi_1^*) = \deg E_1 L = 1/n \deg L \geq c$$ 
Le résultat est alors analogue à la démonstration du lemme $\ref{lemnorm}$.
\end{proof}

\subsection{Classicité}

Énonçons maintenant le théorème de classicité. Soit  $\kappa \in~X(T_M)^+$ ; l'élément $\kappa$ correspond donc à une famille d'entiers
$$\prod_{i=1}^{h}  \prod_{j=1}^{d_i} (k_{1,j,i} \geq \dots \geq k_{g,j,i} ) $$

\begin{theo} \label{theogen}
Soit $f$ une forme surconvergente de poids $\kappa \in X(T_M)^+$ sur $X_{Iw}$, propre pour la famille d'opérateurs de Hecke $U_{p,i}$. Supposons que les valeurs propres $(a_i)$ pour ces opérateurs soient non nulles, et que $\kappa$ vérifie les relations
$$ v(a_i) + \frac{d_i g(g+1)}{2} < \inf_{1 \leq j \leq d_i} k_{g,j,i} $$
pour tout $1\leq i \leq h$. Alors $f$ est classique.
\end{theo}
 
\begin{proof}
La méthode de démonstration permettant d'étendre le résultat du théorème $\ref{maintheo}$ à ce cas s'inspire des travaux de Sasaki ($\cite{Sa}$). \'Ecrivons les différentes étapes de la démonstration. \\
La forme modulaire surconvergente $f$ est sur un ouvert du type Deg$^{-1} (\prod_{1\leq i \leq h} [d_i g -~\varepsilon,d_i g])$ pour un certain $\varepsilon >0$. Nous allons étendre cette forme à $X_{Iw,rig}$ tout entier. Pour cela, nous allons étendre $f$ direction par direction, c'est-à-dire étendre $f$ à $Deg^{-1} ([0,d_1g] \times \prod_{2\leq i \leq h} [d_i g -~\varepsilon,d_i g]  )$. Nous utiliserons pour cela le fait que $f$ est propre pour $U_{p,1}$ et la relation vérifiée par la valeur propre $a_1$. En utilisant l'opérateur $U_{p,2}$, et en répétant le processus, nous allons prolonger $f$ à $Deg^{-1} ([0,d_1g] \times [0,d_2g] \times \prod_{2\leq i \leq h} [d_i g -~\varepsilon,d_i g]  )$, et ainsi de suite, jusqu'à prolonger $f$ à tout $X_{Iw,rig}$. Détaillons le prolongement dans la première direction. \\
\underline{\'Etape $1$ : } Nous étendons la forme modulaire $f$ à l'espace Deg$^{-1} ( ]d_1 g -1,d_1 g ] \times~\prod_{2\leq i \leq h} [d_i g -~\varepsilon, d_i g])$. Nous utilisons pour cela la formule $f=a_1^{-N} U_{p,1}^N f$, et la proposition $\ref{augmente_C}$. Cela permet bien d'étendre $f$ à Deg$^{-1} ( ]d_1 g - 1,d_1 g ] \times~\prod_{2\leq i \leq h} [d_i g -~\varepsilon, d_i g])$. \\
\underline{\'Etape $2$ : } Le théorème $\ref{bigtheo_C}$ permet de définir les séries de Kassaei sur 
$$\mathcal{U} : = Deg^{-1} ( [0,d_1 g - 1 + \alpha ] \times~\prod_{2\leq i \leq h} [d_i g -~\varepsilon, d_i g])$$
où $\alpha$ est un rationnel arbitrairement petit. Le fait que les séries de Kassaei vont converger est assurée par la proposition $\ref{norm_C}$ et la relation vérifiée par $a_1$. Cela permet donc d'étendre $f$ à $\mathcal{U}$. En recollant $f$ avec la forme définie sur $Deg^{-1} ( ]d_1 g - 1,d_1 g ] \times~\prod_{2\leq i \leq h} [d_i g -~\varepsilon, d_i g])$, on voit qu'on peut donc étendre $f$ à $Deg^{-1} ( [0,d_1 g ] \times~\prod_{2\leq i \leq h} [d_i g -~\varepsilon, d_i g])$. \\
$ $
En répétant ce processus, on peut donc étendre $f$ à $X_{Iw,rig}$, c'est-à-dire un élément de $H^0(X_{Iw,rig},\omega^\kappa)$. Le fait que $f$ est classique provient ensuite du principe de Koecher et de GAGA, en utilisant une compactification toroïdale de $X$ (voir le théorème $\ref{toro}$).
\end{proof}

\section{Cas des variétés de Shimura de type~(A)} \label{unitaire}

La méthode de prolongement analytique utilisée s'adapte à d'autres cas. Dans cette partie, nous nous intéressons aux variétés de Shimura de type (A).

\subsection{Données de Shimura}

Rappelons les données paramétrant les variétés de Shimura PEL de type (A) (voir $\cite{Ko}$). Soit $B$ une $\mathbb{Q}$-algèbre simple munie d'une involution positive $\star$. Soit $F$ le centre de $B$ et $F_0$ le sous-corps de $F$ fixé par $\star$. Le corps $F_0$ est une extension totalement réelle de $\mathbb{Q}$, soit $d$ sont degré. Faisons les hypothèses suivantes :

\begin{itemize}
\item $[F:F_0]=2$.
\item Pour tout plongement $F_0 \to \mathbb{R}$, $B \otimes_{F_0} \mathbb{R} \simeq $M$_n(\mathbb{C})$, et l'involution $\star$ est donnée par $A \to \overline{A}^t$.
\end{itemize}

Soit également $(U_{\mathbb{Q}},\langle,\rangle)$ un $B$-module hermitien non dégénéré. Soit $G$ le groupe des automorphismes du $B$-module hermitien $U_{\mathbb{Q}}$ ; pour toute $\mathbb{Q}$-algèbre $R$, on a donc
$$G(R) = \left\{ (g,c) \in GL_{B} (U_{\mathbb{Q}} \otimes_{\mathbb{Q}} R) \times R^* , \langle gx,gy \rangle =c \langle x,y\rangle \text{ pour tout } x,y \in U_{\mathbb{Q}} \otimes_{\mathbb{Q}} R \right\} $$

Soient $\tau_1, \dots, \tau_d$ les plongements de $F_0$ dans $\mathbb{R}$ ; soit également $\sigma_i$ et $\overline{\sigma_i}$ les deux plongements de $F$ dans $\mathbb{C}$ étendant $\tau_i$. Le choix de $\sigma_i$ donne un isomorphisme $F \otimes_{F_0} \mathbb{R} \simeq \mathbb{C}$. On a également $B_i=B \otimes_{F_0,\tau_i} \mathbb{R} \simeq $M$_n (\mathbb{C})$. Notons $U_i = U_\mathbb{Q} \otimes_{F_0,\tau_i} \mathbb{R}$. D'après l'équivalence de Morita, $U_i \simeq~\mathbb{C}^n \otimes W_i$, où $B_i$ agit sur le premier facteur et $W_i$ est un $\mathbb{C}$-espace vectoriel. La structure anti-hermitienne sur $U_i$ en induit une sur $W_i$, et on note $(a_i,b_i)$ sa signature.

\begin{rema}
Si on avait choisi l'isomorphisme $F \otimes_{F_0} \mathbb{R} \simeq \mathbb{C}$ donné par $\overline{\sigma_i}$, on aurait obtenu le couple $(b_i,a_i)$. Ainsi, on dispose pour chaque plongement $\tau_i$ d'un couple d'entiers, défini à permutation près, mais bien défini si on choisit un plongement de $F$ au-dessus de $\tau_i$.
\end{rema}

Alors $G_\mathbb{R}$ est isomorphe à
$$\text{G} \left( \prod_{i=1}^d \text{U}(a_i,b_i) \right)$$
où $a_i + b_i$ est indépendant de $i$ et vaut $\frac{1}{2nd} $dim$_\mathbb{Q} U_\mathbb{Q}$. Nous noterons cette quantité $a+b$. \\
Donnons-nous également un morphisme de $\mathbb{R}$-algèbres $h : \mathbb{C} \to $End$_B U_\mathbb{R}$ tel que $\langle h(z)v,w\rangle ~=~\langle v,h(\overline{z})w\rangle$ et $(v,w) \to \langle v,h(i)w\rangle$ est définie positive. Ce morphisme définit donc une structure complexe sur $U_\mathbb{R}$ : soit $U^{1,0}_{\mathbb{C}}$ le sous-espace de $U_\mathbb{C}$ pour lequel $h(z)$ agit par la multiplication par $z$. On a alors $U^{1,0}_{\mathbb{C}} \simeq \prod_{i=1}^d (\mathbb{C}^n)^{a_i} \oplus \overline{(\mathbb{C}^n)}^{b_i}$ en tant que $B \otimes_{\mathbb{Q}} \mathbb{R} \simeq \oplus_{i=1}^d $M$_n( \mathbb{C})$-module (l'action de M$_n(\mathbb{C})$ sur $(\mathbb{C}^n)^{a_i} \oplus \overline{(\mathbb{C}^n)}^{b_i}$ est l'action standard sur le premier facteur et l'action conjuguée sur le second) . \\
Soient également un ordre $O_B$ de $B$ stable par par $\star$, et un réseau $U$ de $U_\mathbb{Q}$ tel que l'accouplement $\langle,\rangle$ restreint à $U\times U$ soit à valeurs dans $\mathbb{Z}$. Nous ferons également les hypothèses suivantes : 
\begin{itemize}
\item $B \otimes_{\mathbb{Q}} \mathbb{Q}_p$ est isomorphe à un produit d'algèbres de matrices à coefficients dans une extension non ramifiée de $\mathbb{Q}_p$.
\item $O_B$ est un ordre maximal en $p$.
\item L'accouplement $U \times U \to \mathbb{Z}$ est parfait en $p$.
\end{itemize}
$ $\\
Soit $\mathbb{Z}_{(p)}$ le localisé de $\mathbb{Z}$ en $p$ ; $O_B$ est un $\mathbb{Z}_{(p)}$-module libre. Soit $e_1, \dots, e_t$ une base de ce module, et 
$$\text{det}_{U^{1,0}} = f(X_1, \dots, X_t) = \det (X_1 \alpha_1 + \dots + X_t \alpha_t ;U^{1,0}_{\mathbb{C}} \otimes_{\mathbb{C}} \mathbb{C} [X_1, \dots, X_t])$$
On montre ($\cite{Ko}$) que $f$ est un polynôme à coefficients algébriques. Le corps de nombres $E$ engendré par ses coefficients est appelé le corps réflexe. \\
De plus, d'après les hypothèses précédentes, $p$ est non ramifié dans $F_0$. Soient $\pi_1, \dots, \pi_h$ les idéaux de $F_0$ au-dessus de $p$, et $d_i$ le degré résiduel de chacune de ces places. Par hypothèse, $\pi_i$ est non ramifié dans $F$ ; nous sommes donc amenés à distinguer deux cas. 

\begin{itemize}
\item Nous dirons que $\pi_i$ est dans le cas $1$ si $\pi_i$ est totalement décomposé dans $F$. On note dans ce cas $\pi_i^+$ et $\pi_i^-$ les deux idéaux de $F$ au-dessus de $\pi_i$.
\item Nous dirons que $\pi_i$ est dans le cas $2$ si $\pi_i$ est inerte dans $F$.
\end{itemize}

Nous avons alors $O_B \otimes~\mathbb{Z}_p \simeq~\prod_{i=1}^{h} O_{B,i}$, avec

\begin{itemize}
\item si $\pi_i$ est dans le cas $1$, $O_{B,i} = $M$_n(\mathbb{Z}_{p^{d_i}}) \oplus $M$_n(\mathbb{Z}_{p^{d_i}})$, où $\mathbb{Z}_{p^{d_i}}$ est l'anneau des entiers de l'unique extension non ramifiée de degré $d_i$ de $\mathbb{Q}_p$.
\item si $\pi_i$ est dans le cas $2$, $O_{B,i} = $M$_n(\mathbb{Z}_{p^{2d_i}})$.
\end{itemize}

Soit $\Sigma $ l'ensemble des plongements de $F_0$ dans $\overline{\mathbb{Q}_p}$, et $\Sigma_i$ le sous-ensemble de $\Sigma$ formé des plongements envoyant $\pi_i$ dans l'idéal maximal de l'anneau des entiers de $\overline{\mathbb{Q}_p}$. Alors $\Sigma_i$ est de cardinal $d_i$, et $\Sigma$ est l'union disjointe des $\Sigma_i$. Si $\pi_i$ est dans le cas $1$, et si $ \sigma \in \Sigma_i$, il existe deux plongements de $F$ au-dessus de $\sigma$ : $\sigma^+$ et $\sigma^-$, ces plongements étant respectivement au-dessus de $\pi_i^+$ et $\pi_i^-$. Nous ordonnerons le couple $(a_\sigma, b_\sigma)$ de telle sorte que $a_\sigma$ soit associé à $\sigma^+$.

\subsection{Variété de Shimura}

Définissons maintenant la variété de Shimura PEL de type (A) associée à $G$. Soit $K$ une extension finie de $\mathbb{Q}_p$ contenant tous les plongements $F \to \overline{\mathbb{Q}_p}$, et $O_K$ son anneau des entiers. Soit $N \geq 3$ un entier premier à $p$. 

\begin{defi}
Soit $X$ l'espace de modules sur Spec $O_K$ dont les $S$-points sont les classes d'isomorphismes des $(A,\lambda,\iota,\eta)$ où
\begin{itemize}
\item $A \to S$ est un schéma abélien
\item $\lambda : A \to A^t$ est une polarisation de degré premier à $p$.
\item $\iota : O_B \to $End $A$ est compatible avec les involutions $\star$ et de Rosati, et les polynômes $\det_{U^{1,0}}$ et $\det_{Lie (A)}$ sont égaux.
\item $\eta : A[N] \to U/NU$ est une similitude symplectique $O_B$-linéaire, qui se relève localement pour la topologie étale en une similitude symplectique $O_B$-linéaire
$$H_1 (A,\mathbb{A}_f^p) \to U \otimes_{\mathbb{Z}} \mathbb{A}_f^p$$
\end{itemize}
\end{defi}

\begin{prop}
L'espace $X$ est un schéma quasi-projectif sur Spec $O_K$.
\end{prop}

\begin{rema}
Explicitons la condition du déterminant. Notons $St_{O_K}$ est le $O_B \otimes_{\mathbb{Z}} O_K$-module défini par $St_{O_K} = \oplus_{i=1}^h St_{O_K,i}$, où $St_{O_K,i}$ est le $O_{B,i}$-module défini par 
$$St_{O_K,i} = \bigoplus_{\sigma \in \Sigma_i} (O_K^n)^{a_\sigma} \oplus (O_K^n)^{b_\sigma}$$
où dans le cas $1$ l'action de $O_{B,i} = $M$_n(\mathbb{Z}_{p^{d_i}}) \oplus $M$_n(\mathbb{Z}_{p^{d_i}})$ est l'action standard sur chacun des facteurs, et dans le cas $2$ l'action de $O_{B,i} = $M$_n(\mathbb{Z}_{p^{2d_i}})$ est l'action standard, avec $\mathbb{Z}_{p^{2d_i}}$ qui agit par un certain plongement sur le premier facteur, et par son conjugué sur le second (le choix de ce plongement est imposé par l'ordre sur le couple $(a_\sigma, b_\sigma)$). Alors la condition du déterminant est équivalente au fait que Lie$(A)$ soit isomorphe localement pour la topologie de Zariski à $St_{O_K} \otimes_{O_K} \mathcal{O}_S$ comme $O_B \otimes_{\mathbb{Z}} \mathcal{O}_S$-module. On voit donc en particulier que le schéma abélien est de dimension $nd(a+b)$.
\end{rema}

Nous allons maintenant définir une structure de niveau Iwahorique sur $X$. Si $A \to S$ est un schéma abélien avec action de $O_B$, on a donc
$$A[p^\infty] = \oplus_{i=1}^{h} A[\pi_i^\infty]$$
Les groupes de Barsotti-Tate $A[\pi_i^\infty]$ sont principalement polarisés de dimension $nd_i (a+b)$, et muni d'une action $O_{B,i}$. De plus, si $\pi_i$ est dans le cas $1$, alors
$$A[p^\infty] = A[(\pi_i^+)^\infty] \oplus A[(\pi_i^-)^\infty]$$
De plus, les groupes $A[(\pi_i^+)^\infty]$ et $A[(\pi_i^-)^\infty]$ sont des Barsotti-Tate de hauteur $d_i(a+b)$ de dimensions respectives $d_i a_i$ et $d_i b_i$, et munis d'une action de M$_n(\mathbb{Z}_{p^{d_i}})$. Remarquons que ces deux groupes sont duaux l'un de l'autre (cela résulte de la compatibilité entre l'involution de Rosati et la conjugaison complexe). 

\begin{defi} 
Soit $X_{Iw}$ l'espace de modules sur $O_K$ dont les $S$-points sont les $(A,\lambda,\iota,\eta,H_{i,j})$ où $(A,\lambda,\iota,\eta) \in X(S)$ et
\begin{itemize}
\item si $\pi_i$ est dans le cas $1$, $0=H_{i,0} \subset H_{i,1} \subset \dots \subset H_{i,a+b}=A[\pi_i^+]$ est un drapeau de sous-groupes finis et plats de $A[\pi_i^+]$ stables par $O_B$, chaque $H_{i,j}$ étant de hauteur $nd_i j$.
\item si $\pi_i$ est dans le cas $2$, $0=H_{i,0} \subset H_{i,1} \subset \dots \subset H_{i,a+b}=A[\pi_i]$ est un drapeau de sous-groupes finis et plats de $A[\pi_i]$ stables par $O_B$, chaque $H_{i,j}$ étant de hauteur $2nd_i j$, avec $H_{i,a+b-j}$ égal à l'orthogonal de $H_{i,j}$.
\end{itemize}
\end{defi}

Nous noterons $X_{rig}$ et $X_{Iw,rig}$ les espaces rigides associés respectivement à $X$ et $X_{Iw}$.

\begin{rema}
Il pourrait sembler que choisir $A[(\pi_i^+)]$ (au lieu de $A[\pi_i^-]$) dans le cas $1$ rompe la symétrie. Il n'en est en fait rien, puisque ces deux groupes sont duaux l'un de l'autre : $A[\pi_i^+] \simeq A[\pi_i^-]^D$. Cette dualité induit donc un accouplement parfait 
$$\langle,\rangle : A[\pi_i^+] \times A[\pi_i^-] \to \mathbb{G}_m$$
Si $H$ est un sous-groupe de $A[\pi_i^+]$, on peut donc considérer l'orthogonal de $H$ pour cet accouplement, $H^\bot$, qui est un sous-groupe de $A[\pi_i^-]$. Un drapeau $(H_j)$ de $A[\pi_i^+]$ donne donc par orthogonalité un drapeau $(H_{a+b-j}^\bot)$ de $A[\pi_i^-]$. \\
Remarquons également que pour tout sous-groupe $H$ de $A[\pi_i^+]$, on dispose d'un diagramme commutatif

\begin{displaymath}
\xymatrix{
A[\pi_i^-] \ar[r] & A[\pi_i^+]^D \\
H^\bot \ar[r] \ar[u] & (A[\pi_i^+] / H)^D \ar[u]
}
\end{displaymath}
où les flèches horizontales sont des isomorphismes. D'où $H^\bot \simeq (A[\pi_i^+] / H)^D$. 
\end{rema}

\begin{rema}
Les schémas $X$ et $X_{Iw}$ sont en fait défini sur le corps réflexe $E$ ; cela résulte du fait que le $O_B \otimes_{\mathbb{Z}} O_K$-module $St_{O_K}$ est en fait défini sur $E$. Néanmoins, nous devrons nous placer sur $O_K$ pour définir les faisceaux des formes modulaires.
\end{rema}

Pour définir les formes modulaires surconvergentes, nous aurons besoin de faire l'hypothèse que le lieu ordinaire de la variété de Shimura est non vide. D'après un résultat de Wedhorn ($\cite{We}$), cela est équivalent au fait suivant.

\begin{hypo}
Nous supposerons que $p$ est totalement décomposé dans le corps réflexe $E$.
\end{hypo}

Cette hypothèse a les conséquences suivantes sur les couples d'entiers $(a_i,b_i)$.

\begin{prop}
Supposons que $\pi_i$ soit dans le cas $1$. Alors il existe un couple d'entiers $(a_i,b_i)$ tel que $(a_\sigma,b_\sigma) = (a_i,b_i)$ pour tout $\sigma \in \Sigma_i$. Si $\pi_i$ est dans le cas $2$, alors $a_\sigma =~b_\sigma =~(a+b)/2$.
\end{prop}

\begin{proof}
Supposons que $\pi_i$ soit dans le cas $1$. Par hypothèse, il existe un schéma abélien $A$ défini sur une extension finie de $O_K$ tel que le groupe $p$-divisible $A[\pi_i^\infty]$ soit ordinaire, c'est-à-dire extension d'un groupe multiplicatif par un groupe étale. On en déduit que $A[(\pi_i^+)^\infty]$ est également ordinaire. Or $A[\pi_i^+]$ est muni d'une action de $\mathbb{Z}_{p^{d_i}}$, donc d'après la partie $\ref{partial}$, on peut définir ses degrés partiels. Pour tout $\sigma \in \Sigma_i$, on a $\deg_\sigma A[\pi_i^+] = n a_\sigma$. Or le degré d'un groupe étale est nul, et si $G$ est un groupe multiplicatif muni d'une action de $\mathbb{Z}_{p^{d_i}}$, tous ses degrés partiels sont égaux. On en déduit que $a_\sigma$ ne dépend pas de $\sigma$. Il existe donc un entier $a_i$ tel que $a_\sigma = a_i$ pour tout $\sigma \in \Sigma_i$. Si $b_i = a+b-a_i$, on a alors $b_\sigma = a+b-a_\sigma = b_i$ pour tout $\sigma \in \Sigma_i$. \\
Si $\pi_i$ est dans le cas $2$, soit $A$ un schéma abélien défini sur une extension finie de $O_K$ avec $A[\pi_i^\infty]$ ordinaire. Alors $A[\pi_i]$ est muni d'une action de $\mathbb{Z}_{p^{2d_i}}$, et pour tout $\sigma \in \Sigma_i$, si $\sigma_1$ et $\sigma_2$ sont les deux plongements de $F$ au-dessus de $\sigma$, on a $\deg_{\sigma_1} A[\pi_i] = a_\sigma$ et $\deg_{\sigma_2} A[\pi_i] = b_\sigma$, si on suppose que l'entier $a_\sigma$ est associé à $\sigma_1$. Puisque tous les degrés partiels doivent être égaux, on en déduit le résultat.
\end{proof}

Si $\pi_i$ est dans le cas $2$, on notera $a_i=b_i=(a+b)/2$, de telle sorte que pour tout $\sigma \in \Sigma_i$, on a $a_\sigma=a_i$ et $b_\sigma = b_i$ quelque soit le cas.

\subsection{Formes modulaires} \label{defuni}

Soit $A$ le schéma abélien universel sur $X$, et soit $e^* \Omega_{A/X}^1$ le faisceau conormal relatif à la section unité de $A$. D'après ce qui précède, il est localement pour la topologie de Zariski isomorphe à $St \otimes_{\mathbb{Z}_p} \mathcal{O}_X$ comme $O_B \otimes \mathcal{O}_X$-~module, où $St=\oplus_{i=1}^h St_i$, et $St_i$ est le $O_{B,i}$-module défini par 
\begin{itemize}
\item $St_i = (\mathbb{Z}_{p^{d_i}}^n)^{a_i} \oplus (\mathbb{Z}_{p^{d_i}}^n)^{b_i}$ dans le cas $1$
\item $St_i = (\mathbb{Z}_{p^{2d_i}}^n)^{a_i}$ dans le cas $2$
\end{itemize}

Soit 
$$\mathcal{T} = \text{Isom}_{O_B \otimes \mathcal{O}_X} (St \otimes \mathcal{O}_X, e^* \Omega_{A/X}^1)$$
C'est un torseur sur $X$ sous le groupe
$$M=\left ( \prod_{i \in S_1} Res_{\mathbb{Z}_{p^{d_i}} / \mathbb{Z}_p} (GL_{a_i} \times GL_{b_i}) \times \prod_{i \in S_2} Res_{\mathbb{Z}_{p^{2d_i}} / \mathbb{Z}_p} GL_{a_i}  \right) \times_{\mathbb{Z}_p} O_K$$
où $S_j$ est l'ensemble des indices $i$ tels que $\pi_i$ est dans le cas $j$. \\
Soit $B_M$ le Borel supérieur de $M$, $U_M$ son radical unipotent, et $T_M$ son tore maximal. Soit $X(T_M)$ le groupe des caractères de $T_M$, et $X(T_M)^+$ le cône des poids dominants pour $B_M$. Si $\kappa \in X(T_M)^+$, on note $\kappa'=- w_0 \kappa \in X(T_M)^+$, où $w_0$ est l'élément le plus long du groupe de Weyl de $M$ relativement à $T_M$. \\
Soit $\phi : \mathcal{T} \to X$ le morphisme de projection.

\begin{defi}
Soit $\kappa \in X(T_M)^+$. Le faisceau des formes modulaires de poids $\kappa$ est $\omega^\kappa =~\phi_* O_\mathcal{T}[\kappa']$, où $\phi_* O_\mathcal{T}[\kappa']$ est le sous-faisceau de $\phi_* O_\mathcal{T}$ où $B_M=T_M U_M$ agit par $\kappa$ sur $T_M$ et trivialement sur $U_M$.
\end{defi} 

Une forme modulaire de poids $\kappa$ sur $X$ est donc une section globale de $\omega^\kappa$, soit un élément de $H^0(X , \omega^\kappa)$. En utilisant la projection $X_{Iw} \to X$, on définit de même le faisceau $\omega^\kappa$ sur $X_{Iw}$, ainsi que les formes modulaires sur $X_{Iw}$. \\

En utilisant l'équivalence de Morita, une définition équivalente du torseur $\mathcal{T}$ est
$$ \mathcal{T} = \prod_{i \in S_1} \text{Isom}_{\mathbb{Z}_{p^{d_i}} \otimes \mathcal{O}_X} ((\mathbb{Z}_{p^{d_i}}^{a_i} \oplus \mathbb{Z}_{p^{d_i}}^{b_i}) \otimes_{\mathbb{Z}_p} \mathcal{O}_X, E_i \cdot e^* \Omega_{A/X}^1) \times \prod_{i \in S_2} \text{Isom}_{\mathbb{Z}_{p^{2d_i}} \otimes \mathcal{O}_X} (\mathbb{Z}_{p^{2d_i}}^{(a+b)/2}  \otimes_{\mathbb{Z}_p} \mathcal{O}_X, E_i \cdot e^* \Omega_{A/X}^1) $$
où $E_i$ est l'élément de $\prod_{j=1}^h O_{B,j}$ dont toutes les coordonnées sont nulles, sauf la $i$-ième égal à la matrice $E_{1,1} \oplus E_{1,1}$ si $i$ appartient à $S_1$, et égal à la matrice $E_{1,1}$ si $i \in S_2$.  

\begin{rema}
Le poids $\kappa$ d'une forme modulaire est une famille d'entiers
$$\prod_{i=1}^{h}  \prod_{j=1}^{d_i} ((k_{1,j,i} \geq \dots \geq k_{a_i,j,i} ) , (l_{1,j,i} \geq \dots \geq l_{b_i,j,i} ) ) $$
\end{rema}

Comme dans les cas précédent, nous n'avons pas à nous préoccuper des pointes pour la définition des formes modulaires, car il existe des modèles entiers des compactifications toroïdales.

\begin{theo}[$\cite{P-S 1}$ partie $6$] \label{toro_A}
Il existe une compactification toroïdale $\overline{X}_{Iw}$ de $X_{Iw}$ définie sur $O_K$, dépendant d'un choix combinatoire. Le schéma $\overline{X}_{Iw}$ est propre sur $O_K$, et le faisceau $\omega^\kappa$ s'étend à $\overline{X}_{Iw}$. De plus, on a le principe de Koecher algébrique et rigide, c'est-à-dire $H^0(\overline{X}_{Iw}, \omega^\kappa) = H^0 (X_{Iw}, \omega^\kappa)$ et $H^0(\overline{X}_{Iw,rig}, \omega^\kappa) = H^0 (X_{Iw,rig}, \omega^\kappa)$ où $\overline{X}_{Iw,rig}$ désigne l'espace rigide associé à $\overline{X}_{Iw}$.
\end{theo}

Définissons maintenant les fonctions degrés.

\begin{defi}
Soit $i$ un entier entre $1$ et $h$. On définit la fonction $Deg_i : X_{Iw,rig} \to~[0,d_i a_i]$ par $Deg_i((A,\lambda,\iota,\eta,H_{j,k})) = 1/n \deg H_{i,a_i}$. On définit également la fonction $Deg : X_{Iw,rig} \to~\prod_{i=1}^h [0,d_i a_i]$ par $x \to (Deg_i(x))$.
\end{defi}

\begin{rema}
Ici encore, le fait de diviser par $n$ est lié à l'action de l'algèbre de matrices. 
\end{rema}
$ $\\
Soit $X_{Iw,rig}^{mult}$ le lieu ordinaire-multiplicatif de $X_{Iw,rig}$ ; il est égal par définition à 
$$Deg^{-1} (\{d_1 a_1 \} \times \dots \times \{d_h a_h \})$$
Nous pouvons maintenant définir les formes surconvergentes sur pour $X$.

\begin{defi}
L'ensemble des formes modulaires surconvergentes est défini par
$$H^0(X_{Iw,rig},\omega^\kappa)^\dagger := \text{colim}_\mathcal{V} H^0 (\mathcal{V},\omega^\kappa)$$
où la colimite est prise sur les voisinages stricts $\mathcal{V}$ de $X_{Iw,rig}^{mult}$ dans $X_{Iw,rig}$.
\end{defi}

\subsection{Opérateurs de Hecke}

Soit $1 \leq i \leq h$. Soit $C_i$ l'espaces des modules sur $K$ dont les $S$-points sont les $(A,\lambda,\iota,\eta,H_{j,k},L)$ avec $(A,\lambda,\iota,\eta,H_{j,k}) \in X_{Iw}(S)$ et 
\begin{itemize}
\item $L$ un sous-groupe fini et plat de $A[\pi_i^+]$, stable par $O_B$, et supplémentaire générique de $H_{i,a_i}$ dans $A[\pi_i^+]$ dans le cas $1$.
\item $L$ un sous-groupe fini et plat de $A[\pi_i]$, stable par $O_B$, totalement isotrope et supplémentaire générique de $H_{i,a_i}$ dans $A[\pi_i]$ dans le cas $2$.
\end{itemize}

Remarquons que dans le cas $1$ $A[\pi_i^-]$ est la somme directe sur $K$ de $H_{i,a_i}^\bot$ et de $L^\bot$, donc que $A[\pi_i]$ est la somme directe sur $K$ de $H_{i,a_i} \oplus H_{i,a_i}^\bot$ et de $L \oplus L^\bot$. \\ 
Nous avons deux morphismes $p_1, p_2 : C_i \to X_{Iw,K}=X_{Iw} \times K$ : $p_1$ est l'oubli de $L$, et $p_2$ est le quotient par $L_0$, où $L_0$ est égal à $L \oplus L^\bot$ dans le cas $1$ et à $L$ dans le cas $2$. Pour définir la projection $p_2$, nous devons choisir la polarisation sur le schéma abélien $A/L$. On fixe un élément totalement positif $x_i$ de $O_{F_0}$, de valuation $\pi_j$-adique $1$ si $j=i$ et $0$ sinon. On définit la polarisation sur $A/L$ comme la polarisation descendue $x_i \cdot \lambda$. Comme décrit dans la remarque $\ref{non_canonique}$, ce choix ne créera pas de difficulté. Nous noterons encore $p_1,p_2$ les morphismes induits $C_{i,rig} \to X_{Iw,rig}$. Notons également $\pi : A \to A/L_0$ l'isogénie universelle au-dessus de $C_i$. Celle-ci induit un isomorphisme $\pi^* : \omega_{(A/L_0)/X} \to \omega_{A/X}$, et donc un morphisme 
$\pi^* (\kappa) :~p_2^* \omega^\kappa \to~p_1^* \omega^\kappa$. Pour tout ouvert $\mathcal{U}$ de $X_{Iw,rig}$, nous pouvons donc former le morphisme composé
\begin{displaymath}
\widetilde{U}_{p,i} :   H^0(U_p(\mathcal{U}),\omega^\kappa) \to H^0 ( p_1^{-1} (\mathcal{U}), p_2^* \omega^\kappa) \overset{\pi^*(\kappa)}{\to} H^0(p_1^{-1}(\mathcal{U}) , p_1^* \omega^\kappa) \overset{Tr_{p_1}}{\to}  H^0(\mathcal{U},\omega^\kappa)
\end{displaymath}

\begin{defi}
L'opérateur de Hecke agissant sur les formes modulaires est défini par $U_{p,i}~=~\frac{1}{p^{n_i}} \widetilde{U}_{p,i}$ avec $n_i=d_i a_i b_i $. 
\end{defi}

Les propriétés vérifiées par les opérateurs $U_{p,i}$ sont les mêmes que dans les parties précédentes.

\begin{prop}
Soit $x \in X_{Iw,rig}$, et $y \in U_{p,i}(x)$. Alors
\begin{itemize}
\item Si $j \neq i$, alors $Deg_j(y) = Deg_j(x)$.
\item $Deg_i (y) \geq Deg_i(x)$.
\item Si $Deg_i(x) = Deg_i(y)$, alors $Deg_i(x) \in \mathbb{Z}$.
\end{itemize}
\end{prop}

\begin{prop} \label{augmente_A}
Soit $r$ un entier compris entre $0$ et $d_ia_i -1$. Soit $0 < \lambda < \mu < 1$ des réels. Alors, il existe un entier $N$ tel que
$$U_{p,i}^N (Deg_i^{-1}([r + \lambda,d_i a_i])) \subset Deg_i^{-1}([r + \mu,d_i a_i])$$
\end{prop}

Nous avons donc défini $h$ opérateurs agissant sur les formes modulaires. Remarquons également que puisque ces opérateurs augmentent le degré, ils agissent sur les formes surconvergentes. \\
Enfin, nous pouvons également décomposer les opérateurs de Hecke suivant le nombre de bons ou mauvais supplémentaires. Pour simplifier les notations, nous énonçons le théorème pour $U_{p,1}$. Soit $\alpha$ un rationnel avec $0< \alpha <1$, $I_2, \dots, I_h$ des intervalles compacts avec $I_k \subset [0,d_k a_k]$. On pose $\mathcal{U} = Deg^{-1} ([0,d_1a_1 - 1 +\alpha] \times I_2 \times \dots \times I_h)$.

\begin{theo}  \label{bigtheo_A}
Soit $N \geq 1$ et $\beta$ un rationnel avec $0<\beta<1$. Il existe une bonne suite d'ouverts $(\mathcal{U}_i (N))_{i \in S_N}$ de $\mathcal{U}$ , tels que pour tout $i\geq 0$, on peut décomposer la correspondance $U_{p,1}^N$ sur $\mathcal{U}_{i}(N) \backslash \mathcal{U}_{i+1}(N)$ en 
$$ U_{p,1}^N = \left ( \coprod_{k=0}^{N-1} U_{p,1}^{N-1-k} \circ T_k  \right ) \coprod T_N$$
avec $T_0 = U_{p,i,N}^{good}$, pour $0 < k < N$
$$T_k = \coprod_{i_1 \in S_{N-1}, \dots, i_k \in S_{N-k}}  U_{p,1,i_k,N}^{good} U_{p,1,i_{k-1},i_k,N}^{bad} \dots U_{p,1,i,i_1,N}^{bad}$$
et
$$T_N = \coprod_{i_1 \in S_{N-1}, \dots, i_{N-1} \in S_1} U_{p,1,i_{N-1},N}^{bad} U_{p,1,i_{N-2},i_{N-1},N}^{bad} \dots U_{p,1,i,i_1,N}^{bad}$$  
avec
\begin{itemize}
\item les images des opérateurs $U_{p,1,j,N}^{good}$ ($j \in S_k$) sont incluses dans $Deg_1^{-1} ( ] d_1 a_1 - 1 +~\beta, d_1 a_1])$
\item les opérateurs $U_{p,1,i,j,N}^{bad}$ ($i \in S_k$, $j \in S_{k-1})$ et $U_{p,j,N}^{bad}$ ($j \in S_1$) sont obtenus en quotientant par un sous-groupe $L$ de degré supérieur ou égal à $n(1 - \beta)$, et ont donc leurs images incluses dans $Deg_1^{-1}([0, d_1a_1-1 + \beta])$.
\end{itemize}
Enfin, si $\beta'$ est un autre rationnel avec $\beta < \beta' <1$, et si $(\mathcal{U}_i'(N))$ est la bonne suite d'ouverts obtenue pour $\beta'$, alors $\mathcal{U}_i'(N)$ est un voisinage strict de $\mathcal{U}_i(N)$ pour tout $i$.
\end{theo}

Nous avons également une majoration de la norme des opérateurs $U_{p,1}^{bad}$. 

\begin{prop} \label{norm_A}
Soit $T$ un opérateur défini sur un ouvert $\mathcal{U}$, égal à $U_{p,1}$, $U_{p,1}^{good}$ ou $U_{p,1}^{bad}$. On suppose que cet opérateur ne fait intervenir que des supplémentaires génériques $L$ de $H_{g,1}$ avec deg $L \geq nc$, pour un certain $c \geq 0$. Alors
$$ \Vert T \Vert_\mathcal{U} \leq p^{n_1-c (\inf_j (k_{a_1,j,1} + l_{b_1,j,1}))}$$
\end{prop}

\begin{proof}
Soient $x$ un point de $X_{Iw,rig}$, correspondant à un couple $(A,\lambda,\iota,\eta, H_{i,j})$ défini sur $O_{\overline{K}}$ et $L$ un supplémentaire générique de $H_{1,a_1}$ comme dans le définition des opérateurs de Hecke. Si $\pi_1$ est dans le cas $1$, alors $L$ est un supplémentaire générique de $H_{1,a_1}$ dans $A[\pi_1^+]$, et on note $L_0 = L \oplus L^\bot$. Si $\pi_1$ est dans le cas $2$, $L$ est un supplémentaire générique de $H_{1,a_1}$ dans $A[\pi_1]$ et on note $L_0=L$. Alors le morphisme $\pi : A \to A/{L_0}$ donne une suite exacte de $O_{\overline{K}} \otimes_{\mathbb{Z}} O_{B}$-modules
$$0 \to \omega_{A/L_0} \overset{\pi^*}{\to} \omega_A \to \omega_{L_0} \to 0$$
On rappelle que $O_B \otimes_{\mathbb{Z}} \mathbb{Z}_p = \prod_{i=1}^{h} O_{B,i}$, où $O_{B,i}$ est la complétion de $O_B$ en $\pi_i$ ; $O_{B,i}$ est donc égal respectivement à M$_n(\mathbb{Z}_{p^{d_i}}) \oplus $M$_n(\mathbb{Z}_{p^{d_i}})$ ou M$_n (\mathbb{Z}_{p^{2d_i}})$ si $\pi_i$ est dans le cas $1$ ou $2$. Le module $\omega_{A}$ se décompose sous l'action de $O_B$ en $\omega_A = \oplus_{i=1}^h \omega_{A,i}$, et de même pour $\omega_{A/L_0}$ et $\omega_{L_0}$. On a donc des suites exactes
$$0 \to \omega_{A/{L_0},i} \overset{\pi_i^*}{\to} \omega_{A,i} \to \omega_{L_0,i} \to 0$$
Puisque $L \subset A[\pi_1]$, $\pi_i^*$ est un isomorphisme si $i \neq 1$. Pour étudier le morphisme $\pi_1^*$, nous allons distinguer suivant le fait que $\pi_1$ soit dans le cas $1$ ou $2$. Supposons que $\pi_1$ soit dans le cas $1$. Alors $O_{B,1} = $M$_n(\mathbb{Z}_{p^{d_1}}) \oplus $M$_n(\mathbb{Z}_{p^{d_1}})$, et on note $E_1^+$ et $E_1^-$ les projecteurs de $O_{B,1}$ égaux respectivement à $(E_{1,1},0)$ et $(0,E_{1,1})$. On a donc des suites exactes de $O_{\overline{K}} \otimes_{\mathbb{Z}_p} \mathbb{Z}_{p^{d_1}}$-modules
$$0 \to E_1^+ \omega_{A/{L_0},1} \overset{\pi_{1,+}^*}{\to} E_1^+ \omega_{A,i} \to E_1^+ \omega_{L_0,1} \to 0$$
et de même en remplaçant $+$ par $-$. Ces morphismes se décomposent sous l'action de $\mathbb{Z}_{p^{d_1}}$ ; pour tout $1 \leq j \leq d_1$, on a des morphismes
$$0 \to (E_1^+ \omega_{A/{L_0},1})_j \overset{\pi_{1,j,+}^*}{\to} (E_1^+ \omega_{A,i})_j \to (E_1^+ \omega_{L_0,1})_j \to 0$$
et de même en remplaçant $+$ par $-$. On a alors $v(\det \pi_{1,j,+}^*) = (\deg_j L)/n$, où $\deg_j$ est le degré partiel (voir la partie $\ref{partial}$). De plus, on a 
$$v(\det \pi_{1,j,-}^*) = (\deg_j L^\bot)/n = (\deg_j (A[\pi_1^+] / L)^D)/n = (a_1 - (a_1 - \deg_j L))/n = (\deg_j L)/n $$
On en déduit que la norme du morphisme $\omega_{A/L_0}^\kappa \to \omega_A^\kappa$ est majorée par
$$p^{ - (\sum_{1 \leq j \leq h} (k_{a_1,j,1} \deg_j L + l_{b_1,j,1} \deg_j L))/n} \leq p^{- (\deg L \inf_j (k_{a_1,j,1} + l_{b_1,j,1}))/n } \leq p^{-c \inf_j (k_{a_1,j,1} + l_{b_1,j,1})}$$
Supposons que $\pi_1$ soit dans le cas $2$. Alors $L= L_0$ est un supplémentaire de $H_{1,a_1}$ dans $A[\pi_1]$ et $O_{B,1} = $M$_n (\mathbb{Z}_{p^{2d_1}})$. On a alors une suite exacte de $O_{\overline{K}} \otimes_{\mathbb{Z}_p} \mathbb{Z}_{p^{2d_1}}$-modules
$$0 \to E_{1,1} \omega_{A/{L},1} \overset{\pi_1'^*}{\to} E_{1,1} \omega_{A,1} \to E_{1,1} \omega_{L,1} \to 0$$
Ces morphismes se décomposent sous l'action de $\mathbb{Z}_{p^{d_1}}$ ; pour tout $1 \leq j \leq d_1$, on a des morphismes
$$0 \to (E_{1,1} \omega_{A/{L},1})_j \overset{\pi_{1,j}^*}{\to} (E_{1,1} \omega_{A,1})_j \to (E_{1,1} \omega_{L,1})_j \to 0$$ 
Chacun de ces modules se décompose pour l'action de $\mathbb{Z}_{p^{2d_1}}$, et on a deux suites exactes
$$0 \to (E_{1,1} \omega_{A/{L},1})_j^+ \overset{\pi_{1,j,+}^*}{\to} (E_{1,1} \omega_{A,1})_j^+ \to (E_{1,1} \omega_{L,1})_j^+ \to 0$$ 
et de même en remplaçant $+$ par $-$. Soit $\sigma_j$ le plongement de $\mathbb{Q}_{p^{2d_1}}$ dans $\overline{\mathbb{Q}}_p$ associé au signe $+$, et $\overline{\sigma}_j$ celui associé au signe $-$. On a $v(\det \pi_{1,j,+}^*) =~(\deg_{\sigma_j} L)/n$ et $v(\det \pi_{1,j,-}^*) =~(\deg_{\overline{\sigma}_j} L)/n$. Or le groupe $L$ est totalement isotrope maximal pour l'accouplement de Weil. Étant donné la compatibilité entre l'involution de Rosati et la conjugaison complexe, on a donc $L \simeq (A[\pi_1]/L)^{D,c}$, où $c$ veut dire que l'action de $O_B$ est tordue par la conjugaison complexe. Si $\sigma$ est un plongement de $F_i$ dans $\overline{\mathbb{Q}}_p$, on a donc $\deg_\sigma L = \deg_{\overline{\sigma}} ((A[\pi_1]/L)^D) = \deg_{\overline{\sigma}} L$. On en déduit que la norme du morphisme $\omega_{A/L}^\kappa \to \omega_A^\kappa$ est majorée par
$$p^{ - (\sum_{1 \leq j \leq h} (k_{a_1,j,1} \deg_{\sigma_j} L + l_{b_1,j,1} \deg_{\overline{\sigma}_j} L))/n} \leq p^{- (\deg L \inf_j (k_{a_1,j,1} + l_{b_1,j,1}))/n } \leq p^{-c \inf_j (k_{a_1,j,1} + l_{b_1,j,1})}$$
Dans les deux cas, on obtient la majoration voulue.
\end{proof}

\subsection{Classicité}

Soit $\kappa \in X(T_M)^+$ le poids d'une forme modulaire ; il correspond à une famille d'entiers

$$\prod_{i=1}^{h}  \prod_{j=1}^{d_i} ((k_{1,j,i} \geq \dots \geq k_{a_i,j,i} ) , (l_{1,j,i} \geq \dots \geq l_{b_i,j,i} ) ) $$

Énonçons maintenant le théorème de classicité pour les variétés de Shimura de type (A).

\begin{theo} \label{theogenuni}
Soit $f$ une forme surconvergente de poids $\kappa$ sur $X_{Iw}$, propre pour la famille d'opérateurs de Hecke $U_{p,i}$. Supposons que les valeurs propres $(a_i)$ pour ces opérateurs vérifient les relations
$$ v(a_i) + d_i a_i b_i < \inf_{1 \leq j \leq d_i} (k_{a_i,j,i} + l_{b_i,j,i})$$
pour tout $1 \leq i \leq h$. Alors $f$ est classique.
\end{theo}
 
\begin{proof}
La méthode de démonstration est identique à celle des variétés de Shimura de type (C). Ecrivons rapidement les différentes étapes. \\
La forme modulaire surconvergente $f$ est sur un ouvert du type Deg$^{-1} (\prod_{1\leq i \leq h} [d_i a_i -~\varepsilon,d_i a_i])$ pour un certain $\varepsilon >0$. Nous allons étendre cette forme à $X_{Iw,rig}$ tout entier. Pour cela, nous allons étendre $f$ direction par direction, c'est-à-dire étendre $f$ à \\
$Deg^{-1} ([0,d_1a_1] \times \prod_{2\leq i \leq h} [d_i a_i -~\varepsilon,d_i a_i]  )$. Nous utiliserons pour cela le fait que $f$ est propre pour $U_{p,1}$ et la relation vérifiée par la valeur propre $a_1$. En utilisant l'opérateur $U_{p,2}$, et en répétant le processus, nous allons prolonger $f$ à $Deg^{-1} ([0,d_1a_1] \times [0,d_2a_2] \times \prod_{2\leq i \leq h} [d_i a_i -~\varepsilon,d_i a_i]  )$, et ainsi de suite, jusqu'à prolonger $f$ à tout $X_{Iw,rig}$. Détaillons le prolongement dans la première direction. \\
\underline{\'Etape $1$ : } Nous étendons la forme modulaire $f$ à l'espace Deg$^{-1} ( ]d_1 a_1 -1,d_1 a_1 ] \times~\prod_{2\leq i \leq h} [d_i a_i -~\varepsilon, d_i a_i])$. Nous utilisons pour cela la formule $f=a_1^{-N} U_{p,1}^N f$, et la proposition $\ref{augmente_A}$. Cela permet bien d'étendre $f$ à Deg$^{-1} ( ]d_1 a_1 - 1,d_1 a_1 ] \times~\prod_{2\leq i \leq h} [d_i a_i -~\varepsilon, d_i a_i])$. \\
\underline{\'Etape $2$ : } Le théorème $\ref{bigtheo_A}$ permet de définir les séries de Kassaei sur \\
$\mathcal{U} : = Deg^{-1} ( [0,d_1 a_1 - 1 + \alpha ] \times~\prod_{2\leq i \leq h} [d_i ai -~\varepsilon, d_i a_i])$, où $\alpha$ est un rationnel arbitrairement petit. Le fait que les séries de Kassaei vont converger est assurée par la proposition $\ref{norm_A}$ et la relation vérifiée par $a_1$. Cela permet donc d'étendre $f$ à $\mathcal{U}$. En recollant $f$ avec la forme définie sur $Deg^{-1} ( ]d_1 a_1 - 1,d_1 a_1 ] \times~\prod_{2\leq i \leq h} [d_i a_i -~\varepsilon, d_i a_i])$, on voit qu'on peut donc étendre $f$ à $Deg^{-1} ( [0,d_1 a_1 ] \times~\prod_{2\leq i \leq h} [d_i a_i -~\varepsilon, d_i a_i])$. \\
$ $
En répétant ce processus, on peut donc étendre $f$ à $X_{Iw,rig}$, c'est-à-dire un élément de $H^0(X_{Iw,rig},\omega^\kappa)$. Le fait que $f$ est classique provient ensuite du principe de Koecher et de GAGA, en utilisant une compactification toroïdale de $X$ (voir le théorème $\ref{toro_A}$).
\end{proof}

\bibliographystyle{amsalpha}

\end{document}